\documentclass[11pt]{amsart}
\usepackage{amsmath,amsthm,verbatim,amssymb,amsfonts,amscd, graphicx}
\usepackage{graphics}
\usepackage{tensor}
\usepackage{enumitem}
\usepackage[utf8]{inputenc}
\usepackage{graphicx}
\usepackage{tikz}
\usepackage{xcolor}

\numberwithin{equation}{section}

\theoremstyle{plain}
\newtheorem{theorem}{Theorem}[section]
\newtheorem{corollary}[theorem]{Corollary}
\newtheorem{lemma}[theorem]{Lemma}
\newtheorem{remark}[theorem]{Remark}

\newtheorem{proposition}[theorem]{Proposition}

\theoremstyle{definition}
\newtheorem{definition}[theorem]{Definition}
\theoremstyle{definition}

\def\sup{\operatorname{sup}}

\def\Rm{\operatorname{Rm}}

\def\sup{\operatorname{sup}}

\def\Rm{\operatorname{Rm}}

\def\pr{\partial}
\def\Lap{\Delta}
\def\Hess{\operatorname{Hess}}
\usepackage{hyperref}

\def\R{\mathbb{R}}

\def\sup{\operatorname{sup}}

\def\Rm{\operatorname{Rm}}

\def\sup{\operatorname{sup}}

\def\Rm{\operatorname{Rm}}

\def\pr{\partial}
\def\Lap{\Delta}
\usepackage{hyperref}

\author{Sven Hirsch}
\address{Institute for Advanced Studies, 1 Einstein Drive, Princeton, NJ 08540, USA}
\email{sven.hirsch@ias.edu}

\author{Jonathan J. Zhu}
\address{Department of Mathematics, University of Washington, Seattle, WA, USA}
\email{jonozhu@uw.edu}

\title{Uniqueness of Blowups for Forced Mean Curvature Flow}

\begin{document}
\maketitle

\begin{abstract}
We prove uniqueness of tangent cones for forced mean curvature flow, at both closed self-shrinkers and round cylindrical self-shrinkers, in any codimension. The corresponding results for mean curvature flow in Euclidean space were proven by Schulze and Colding-Minicozzi respectively. We adapt their methods to handle the presence of the forcing term, which vanishes in the blow-up limit but complicates the analysis along the rescaled flow. 
Our results naturally include the case of mean curvature flows in Riemannian manifolds. 
\end{abstract}

\section{Introduction}

Uniqueness of blowups is a fundamental question in the singularity analysis of various geometric partial differential equations. The most important notion of blowup concerning the formation of singularities in geometric flows is the tangent flow - a limit of rescalings about a fixed spacetime point. For mean curvature flow (MCF) of submanifolds $M_s^n$ in Euclidean space, Schulze \cite{schulze2014uniqueness} proved the uniqueness of this limit if the singularity is modelled on a compact shrinking soliton, and recently Colding-Minicozzi \cite{colding2015uniqueness, colding2019regularity} proved uniqueness if the model is a round cylinder. The latter opened the door to a rich regularity theory for mean curvature flows with `generic' singularities in $\mathbb{R}^N$ \cite{CMsurvey}. 

In this paper, we study \textit{forced} MCF. A mean curvature flow with forcing (MCFf) is a family of submanifolds $M_s^n \subset \mathcal{U}^N \subset \mathbb{R}^N$ which evolve by
\begin{equation*}
\frac{dx}{ds} = \mathbf{H} + \mathbf{F}^\perp,
\end{equation*}
where $\mathbf{H}$ is the mean curvature vector and $\mathbf{F} : \mathcal{U} \to \mathbb{R}^N$ is a smooth, ambient vector field. Note that by isometric embedding, a MCF in an ambient Riemannian manifold $(N,g)$ may be locally considered as a MCF with forcing. Upon rescaling $(N,g)$ will resemble Euclidean space. Similarly, if $\mathbf{F}$ is bounded then it vanishes in the blowup limit, that is, the limiting singularity model is a Euclidean soliton. However, as uniqueness concerns the convergence properties of the sequence, not just the limiting model, it is not clear that uniqueness follows simply from the results of Schulze and Colding-Minicozzi. Our main focus is the cylindrical case: 

\begin{theorem}\label{thm:main}
Let $\mathcal U$ be an open subset of $\mathbb R^N$.
Let $M_t^n$ be a smooth, embedded MCF with forcing in $\mathcal{U} \subset \mathbb{R}^N$. If one tangent flow at at a singular point is a multiplicity one cylinder, then the tangent flow at that point is unique. That is, any other tangent flow is also a cylinder (with the same axis and multiplicity one). 
\end{theorem}

We also cover the compact case, which is somewhat simpler, but will also be instructive of the ultimate strategy for proving uniqueness:

\begin{theorem}\label{thm:compact}
Let $M_t^n$ be a smooth, embedded MCF with forcing in $\mathcal{U} \subset \mathbb{R}^N$. If one tangent flow at at a singular point is a smooth closed shrinker $\Gamma$ with multiplicity 1, then the tangent flow at that point is unique. That is, any other tangent flow is also induced by $\Gamma$ (with multiplicity 1). 
\end{theorem}

In both cases we apply the general method of deriving uniqueness from a {\L}ojasiewicz inequality for the rescaled flow. In the compact case we are able to use the Simon-{\L}ojasiewicz inequality due to Schulze \cite{schulze2014uniqueness}, while in the cylindrical case we prove a new {\L}ojasiewicz-type inequality following the methods of Colding-Minicozzi \cite{colding2015uniqueness, colding2019regularity}. We remark that our methods are fairly general and should also apply, for instance, to the class of singularity models studied by the second named author in \cite{zhu2020ojasiewicz}.

\subsection{Background and history}

Geometric flows have led to many striking results in topology, geometry and general relativity during the last decades, including proofs of the Poincar\'e conjecture \cite{perelman2002entropy}, the Differentiable Sphere Theorem \cite{brendle2009manifolds}, and the Riemannian Penrose Inequality \cite{huisken2001inverse}.

\emph{Mean curvature flow} (MCF) is the parabolic analog of minimal surfaces.
Apart from their intrinsic appeal, minimal surfaces had many geometric applications and contributed to our understanding of manifolds with lower bounds on their curvature.
Recently, surfaces of prescribed mean curvature have attracted much attention.
They arise naturally as isoperimetric surfaces and $\mu$-bubbles and led to several scalar curvature results which previously have been inaccessible via the classical minimal surface or Dirac operator methods \cite{chodosh2020generalized, gromov2020no}.

The parabolic analog of surfaces of prescribed mean curvature is mean curvature flow with forcing (MCFf).
In view of Nash's embedding theorem, MCFf in higher codimension also generalizes MCF in Riemannian manifolds.
For instance, MCF in Riemannian manifolds recently led to Urysohn width and waist inequalities \cite{liokumovich2020waist}.

Most geometric PDE exhibit singular behavior, and it is of great importance to better understand these singularities. Typically, this is done by rescaling arguments, and an important question is the uniqueness of blowups at a singular point. For stable minimal surfaces, this has been resolved by Simon who showed in the pioneering work \cite{simon1983asymptotics} uniqueness of tangent cones. The foundation of Simon's proof is an infinite dimensional \emph{{\L}ojasiewicz inequality} which he established using Lyapunov-Schmidt reduction.

{\L}ojasiewicz inequalities have been a very active area of research the last years. For MCF, Schulze applied Simon's work to prove a {\L}ojasiewicz inequality near compact shrinkers. More recently, Colding-Minicozzi proved {\L}ojasiewicz-type inequalities near the round cylinder in Euclidean space. As mentioned above, these results were used to prove uniqueness of tangent flows in the respective cases. For other results on {\L}ojasiewicz inequalities for geometric PDE, the reader may consult for instance, \cite{CMein, deruelle2020ojasiewicz, F19, F20, sun2020rigidity, zhu2020ojasiewicz}. 

With uniqueness at cylindrical tangent flows in hand, Colding-Minicozzi were able to develop a regularity theory for MCF in Euclidean space with cylindrical singularities, including sharp estimates on the singular set and regularity results for the arrival time \cite{CMsing, CMdiff, CMreg, CMarnold}. We expect that, as a consequence of our results here, the corresponding results also hold for MCFf and for MCF in arbitrary manifolds which encounter only cylindrical singularities (in particular for mean convex MCF).

\subsection{Proof strategy}

Let us give a brief description of the proofs of Theorem \ref{thm:main} and Theorem \ref{thm:compact}, beginning with Theorem \ref{thm:compact} as it is indicative of the general ``direct" method for uniqueness.

For MCF (without forcing), one observes that MCF corresponds to the gradient flow for the area functional. Moreover, the rescaled flow $\Sigma_t = e^{t/2}M_s$, $t=-\ln(-s)$ is the gradient flow for the Gaussian area $F(\Sigma^n) = (4\pi)^{-n/2} \int_\Sigma e^{-\frac{|x|^2}{4}}$. Uniqueness of the tangent flow to $M_s$ at $(0,0)$ is equivalent to uniqueness of the $t\to \infty$ limit of the rescaled flow. 

The critical points of $F$ are so-called shrinkers, which satisfy the elliptic PDE $\phi:= \mathbf{H} + \frac{x^\perp}{2}=0$, where $\mathbf H$ is the mean curvature vector. Using Simon's {\L}ojasiewicz inequality, Schulze proved an inequality bounding the oscillation of $F$ by a power of $\|\phi\|_{L^2}$ near a compact shrinker $\Gamma$. A key lemma is that surfaces initially close to a compact shrinker remain close forwards in time. Using these ingredients and an inductive argument, Schulze proved a differential inequality for $F(\Sigma_t) - F(\Gamma)$, the solution of which yields a rate of convergence, and in particular implies uniqueness. 

MCF with forcing is not the gradient flow of $F$, and so in our setting $F$ is no longer monotone. Instead, we perturb $F$ to obtain a new monotone quantity $\tilde F_t$, and also prove a stability lemma for almost Brakke flows close to a shrinker. To use Schulze's Simon-{\L}ojasiewicz inequality, we compare $\tilde{F}_t$ to $F$, which results in an additional term in the resulting differential inequality. Fortunately, the error term is exponential, so we can complete the argument if we wait until a large initial time. 

For cylindrical singularities, Colding-Minicozzi \cite{colding2015uniqueness, colding2019regularity} introduced several key innovations to deal with the significant problem of a noncompact limiting object. They developed a method to directly prove {\L}ojasiewicz inequalities by iterated \emph{improvement} and \emph{extension}. Their `improvement step' can be thought of as a {\L}ojasiewicz inequality for surfaces close enough to a cylinder on a large enough set. Their `extension step', on the other hand, extends the closeness to the cylinder in \textit{space} (we well as time); this also has the effect of reducing error terms in the improvement step. By another inductive argument, they are able to prove a recurrence or \textit{discrete} differential inequality for $F$, the solution of which implies uniqueness. 

In our setting of MCF with forcing we encounter again several difficulties related to the loss of gradient flow structure. Actually, even for MCF (without forcing), we also have to deal with some loss of monotonicity when working locally, due to the noncompactness of the cylinder. One of the main components of this paper is in proving a suitable `extension step'. The argument relies on several monotonicity-type estimates to compare the flow at different points in spacetime, which is complicated by the lack of monotonicity for $F$. It also relies on White's version of Brakke regularity for almost Brakke flows, and higher order interior curvature estimates for such flows. A proof of the latter is also included as, to the best of the authors' knowledge, it is not yet in the literature. 

Following the Colding-Minicozzi method, we then combine our extension step with the Colding-Minicozzi `improvement step' to prove a scale comparison theorem, which relates the `cylindrical scale' (that spatial scale on the rescaled flow is close to a cylinder) with the `shrinker scale' defined by $e^{-R_T^2/2} = \int_{T-1}^{T+1} \|\phi\|_{L^2}^2 dt$. However, due to the localization and other error terms, we have to modify the shrinker scale by an exponential error term. It turns out that this error, even after being compounded in both space and time, is small enough that the discrete differential inequality (for the modified functional) still gives a good rate of convergence, and hence uniqueness. For the final uniqueness, note that we adapt the arguments of \cite{colding2015uniqueness} based on the rigidity of the cylinder, rather than the arguments of \cite{colding2019regularity}. 

\subsection*{Overview of the paper}

In Section \ref{s:preliminaries} we establish our notation as well as our notion of rescaled flow, which is used throughout the paper. We also prove certain area bounds which replace entropy-monotonicity. We are then able to immediately prove Theorem \ref{thm:compact} in Section \ref{sec:compact}. The reader may consider this a lighter introduction to the proof strategy used for the later cylindrical case. 

In Section \ref{sec:extension}, we prove our `extension step' for graphs over a sufficiently large portion of the cylinder. This is combined in Section \ref{sec:scale-comparison} with Colding-Minicozzi's `improvement step' to compare the cylindrical scale with our modified shrinker scale. The cylindrical uniqueness Theorem \ref{thm:main} is proven in Section \ref{sec:uniqueness-cyl}, which also contains certain technical modifications of Section \ref{sec:compact} to handle the noncompact case. 

Appendix \ref{A: discrete} deals with the solution of the discrete differential inequality while Appendix \ref{sec:evol-phi} contains a calculation of the evolution of $\phi$ along MCFf.  Finally, Appendix \ref{sec:interior-estimates} handles interior estimates for MCFf, in the spirit of Ecker-Huisken \cite{ecker1991interior}. 

\subsection*{Acknowledgements:} This work was initiated while the first author visited Princeton University and he is grateful to the math department's hospitality.
SH was supported in part by the National Science Foundation under Grant No. DMS-1926686, and by the IAS School of Mathematics.
JZ was supported in part by the National Science Foundation under grant DMS-1802984 and the Australian Research Council under grant FL150100126.

\section{Preliminaries}\label{s:preliminaries}

\subsection{Notation}

We mainly consider submanifolds $\Sigma^n\subset \mathbb{R}^N$. For a vector $v$ we denote by $v^T$ and $v^\perp= \Pi(v)$ the components tangent and normal to $\Sigma$, respectively. 

We define the mean curvature vector to be the negative trace of the second fundamental form, $\mathbf{H} = -A_{ii}$. The \textit{shrinker mean curvature} is $\phi = \mathbf{H} +\frac{x^\perp}{2}$. 

The (spatial) $L^2$-norm will always be weighted by the Gaussian $\rho(x) = (4\pi)^{-n/2}\exp(-\frac{|x|^2}{4})$. 

Given a submanifold $\Sigma$ and a vector field $U$, we may consider the graph $\Sigma_U := \{x+U(x) | x\in \Sigma\}$. We call this a normal graph if $U$ is a normal vector field on $\Sigma$. When the base $\Sigma$ is clear from context, we write $\phi_U$ for the shrinker quantity associated to the normal graph $\Sigma_U$.

\begin{definition}
We say that $\Sigma$ is $(C^{2,\alpha},\epsilon)$-close to $\Gamma$ if $\Sigma$ may be written as the graph of a normal vector field $U$ over (a subset of) $\Gamma$ with $\|U\|_{C^{2,\alpha}}\leq \epsilon$. We say that $\Sigma$ is $(C^{2,\alpha},\epsilon)$-close to $\Gamma$ on $B_R$ if $\Sigma\cap B_R$ is $(C^{2,\alpha},\epsilon)$-close to $\Gamma$. 
\end{definition}

\subsection{Forced flows and rescaling}
\label{sec:localisation-rescaling}

Fix once and for all $r_0>0$.
We will always assume that $M_s$ is a MCF with forcing (MCFf) in $B_{4r_0}$, that is, $M_s$ is a one-parameter family of submanifolds with no boundary in $B_{4r_0} \subset \mathbb{R}^N$, which satisfy $\frac{dx}{ds} = \mathbf{H} + \mathbf{F}^\perp$. We assume $\|\mathbf{F}\|_{C^k} \leq K$ is uniformly bounded on $B_{4r_0}$. In particular $M_s$ is a $K$-almost Brakke flow in $B_{4r_0}$.

The corresponding rescaled flow (which we abbreviate RMCFf) is $\Sigma_t := e^{t/2} M_s$, $t=-\ln(-s)$, and (up to reparametrisation) satisfies $\frac{dx}{dt} = \phi + e^{-t/2}\mathbf{G}^\perp$. Here $\mathbf{G}(x,t) = \mathbf{F}(e^{t/2}x, s)$. 

Throughout this paper, a RMCFf will always be a flow obtained by rescaling a MCFf as above. 

To investigate uniqueness of tangent flows at $s=0$, we need only consider a short time interval $[s_*,0]$ beforehand, $|s_*|\ll1$, and in particular we can assume $\sup_{[-s_*,0]} \mathcal{H}^n(M_s) \leq \mu$ for some $\mu<\infty$.
In particular, we only need to consider the rescaled flow $\Sigma_t$ for $t\gg1$.

\begin{remark}
To prove uniqueness, one ultimately needs to control the velocity $\tilde{\phi} = \phi + e^{-t/2}\mathbf{G}^\perp$ of the rescaled flow, which differs from the shrinker mean curvature $\phi$ by a forcing term. We have chosen to state our estimates for $\phi$, to be consistent with the {\L}ojasiewicz inequalities (which do not involve a flow), with the trade-off of being less direct in estimating the velocity $\tilde{\phi}$.
\end{remark}

\subsection{Gaussian area functionals}

Let $\rho_{y,s}(x) = (4\pi s)^{-n/2} \exp(-\frac{|x-y|^2}{4s})$ and $\Phi_{y,\sigma}(x,s) = \rho_{y,\sigma-s} (x)$. 
The usual $F$-functionals are $F_{y,\sigma}(\Sigma) = \int_\Sigma \rho_{y,\sigma}$, with the distinguished functional $F=F_{0,1}$. The entropy of a submanifold $\Sigma$ measures its geometric complexity and is defined as $\lambda(\Sigma) = \sup_{y\in \mathbb{R}^N, \sigma>0} F_{y,\sigma}(\Sigma)$. 
The normalization of $F$ ensures that $\lambda(\R^n\subset \R^N)=1$.

\subsection{Almost monotonicity and area bounds}
\label{sec:almost-monotonicity}

Fix once and for all a smooth cutoff function $0\leq \psi\leq 1$ such that $\psi=1$ in $B_{3r_0}$ and $\psi=0$ outside $B_{4r_0}$, with $r_0 |D\psi| + r_0^2 |D^2\psi| \leq K_\psi$. 

For unforced MCF, the monotonicity of the Colding-Minicozzi entropy (derived from Huisken's monotonicity formula) provides uniform area growth bounds in terms of area bounds on the initial slice. 

For MCF with forcing, Huisken's monotonicity no longer holds. Instead, for MCFf as above, we derive area bounds for $M_t\cap B_{2r_0}$ based on almost-monotonicity formulae. For any submanifold $M^n$ define 
\begin{equation*}
F^\psi_{y,\sigma} (M) = \int_M \psi\rho_{y,\sigma}.
\end{equation*} 
Then 
\begin{align*}
F_{y,\sigma}(M \cap B_{3r_0}) \leq F^\psi_{y,\sigma}(M) \leq F_{y,\sigma}(M\cap B_{4r_0}).
\end{align*}

Note that $|D\Phi_{y,\sigma}| \leq \frac{|x-y|}{2(\sigma-s)}$. Following the calculations of Ilmanen \cite[Proof of Lemma 7]{ilmanen1995singularities} and White \cite[Sections 10-11]{white1997stratification} we have, for any $y\in B_{r_0}$ and $\sigma>s$, the almost monotonicity formula 
\begin{align*}
\frac{d}{dt} F^\psi_{y,\sigma-s}(M_s) +\frac{1}{2} \int_{M_s}\psi Q^2 \Phi_{y,\tau} \leq \frac{K^2}{2} F^\psi_{y,\sigma-s}(M_s) + \left(\frac{1}{16r_0^2} + \frac{K_\psi}{\sigma-s}\right) \int_{M_s} \Phi_{y,\sigma} \mathbf{1}_{B_{4r_0}\setminus B_{3r_0}}
\end{align*}
where
\begin{align*}
Q=\left| H  +\frac{(x-y)^\perp}{2(\sigma-s)}  -\frac{(D\psi)^\perp}{\psi} \right|.
\end{align*}
Note that $2r_0\leq|x-y|\leq 5r_0$  for $x\in B_{4r_0}\setminus B_{3r_0}$. So the last term is bounded by \[\mu \left(\frac{1}{16r_0^2} + \frac{K_\psi}{\sigma-s} \right) (4\pi(\sigma-s))^{-n/2}e^{-\frac{r_0^2}{\sigma-s}}\] where $\mu$ is the global area bound as in Section \ref{sec:localisation-rescaling}. White notes that this is bounded for $|\sigma-s|\lesssim r^2$, but in fact if we set $z = \frac{\sigma-s}{r_0^2}$, then the error term is given by $(4\pi)^{-n/2}\mu r_0^{-n-2} (\frac{1}{16}+K_\psi z^{-1}) z^{-n/2} e^{-1/z}$. The latter is bounded by $\gamma:=\mu r_0^{-n-2} c(K_\psi,n)$ for all $z>0$.

As in White \cite[Proposition 11]{white1997stratification} this gives the almost monotonicity:

\begin{lemma}\label{J mono}
Let $y\in B_{r_0}$. Then the quantity 
\begin{align*}
J_{y,\bar{\sigma}}(s):= e^{\frac{K^2}{2}(\bar{\sigma}-s)} F^\psi_{y,\bar{\sigma}-s}(M_t) + \frac{2\gamma}{K^2}(e^{\frac{K^2}{2}(\bar{\sigma}-s)}-1)
\end{align*}
 is non-decreasing for $s_*<s<\sigma$. 
\end{lemma}

Given $y\in B_{r_0}$ and $s\in [s_*,0]$ and $\sigma>0$, choose $\bar{\sigma}= \sigma+s$. Then $J_{y,\sigma}(s) \leq J_{y,\sigma}(s_*)$ yields that 
\begin{align*}
F^\psi_{y,\sigma}(M_s) \leq e^{\frac{K^2}{2} s_*} F^\psi_{y,\bar{\sigma}-s_*}(M_{s_*}) + \gamma (e^{\frac{K^2}{2}(s-s_*)}-1).
\end{align*}

\begin{corollary}
\label{cor:area-bound}
For $t\in [t_*,0]$ we have 
\begin{align*}
\sup_{y\in B_{r_0}, \sigma>0} F_{y,\sigma}(M_t \cap B_{3r_0}) \leq e^{\frac{K^2}{2} s_*} \lambda(M_{s_*} ) + \gamma (e^{\frac{K^2}{2}(s-s_*)}-1).
\end{align*}
 In particular, for small enough $s_*$ depending only on $K$, we have 
\begin{align}\label{eq:entropy bound implies area bound}
\sup_{y\in B_{r_0}, \sigma>0} F_{y,\sigma}(M_s \cap B_{3r_0})  \leq 2 \lambda(M_{s_*} ) + 2\gamma.
\end{align}
\end{corollary}



\section{Uniqueness in the compact case}
\label{sec:compact}

In this section, we describe the proof of uniqueness for the compact case, Theorem \ref{thm:compact}. This will also illustrate the overall strategy and some main issues, which also need to be addressed in the non-compact setting. Throughout this section, we consider a RMCFf of closed submanifolds $\Sigma_t$ as in Section \ref{sec:localisation-rescaling}. 

We may assume that $r_0$ is small enough that the sphere $\pr B_{4r_0}$ is a barrier; that is, any closed MCFf that is initially inside $B_{4r_0}$ remains inside $B_{4r_0}$.

\subsection{Almost-monotonicity controls $\phi$}

Recall
\begin{align*}
F(\Sigma_t)=(4\pi)^{-n/2}\int_{\Sigma_t}e^{-\frac{|x|^2}4}=\int_{\Sigma_t}\rho.
\end{align*}
A straightforward calculation shows that $\phi=\mathbf{H}+\frac{x^\perp}{2}$ is precisely the $L^2$-gradient of $F$, in particular
\begin{align*}
\partial_t F(\Sigma_t)=&-\int_{\Sigma_t}\rho \langle\phi,\phi+e^{-t/2}\mathbf{G}^\perp\rangle. 
\end{align*}
We estimate 
$
e^{-t/2} |\phi||\mathbf{G}| \leq \frac{1}{4}|\phi|^2 + e^{-t}K^2.
$
Therefore,
\begin{align*}
\partial_t F(\Sigma_t)\le&-\int_{\Sigma_t}\rho\left(\frac34|\phi|^2-e^{-t}K^2\right)=K^2e^{-t}F(\Sigma_t)-\frac34\int_{\Sigma_t}\rho|\phi|^2. 
\end{align*}
Let $\mu(t)=e^{K^2e^{-t}}$ and define the modified functional
\begin{align*}
\tilde{F}(t):=\mu(t)F(\Sigma_t).
\end{align*}
Note that the modification only depends on $K$. The almost monotonicity becomes a genuine monotonicity for $\tilde{F}$; in particular,
\begin{align}
\label{tilde F monotonicity}
\partial_t\tilde{F}\leq -\frac34\mu(t)\int_{\Sigma_t}\rho|\phi|^2,
\end{align}
and hence
\begin{align}
\label{L2 estimate}
\int_{t_1}^{t_2}dt \int_{\Sigma_t}|\phi|^2\rho\le 2(\tilde F(t_1)-\tilde F(t_2)).
\end{align}

\subsection{{\L}ojasiewicz inequality and differential inequality}

Recall Schulze's {\L}ojasiewicz-Simon inequality \cite{schulze2014uniqueness} (also see \cite[Appendix A]{CMwand}): 

\begin{theorem}
\label{thm:compact-lojasiewicz}
If $\Gamma$ is a closed shrinker then there exists $C,\epsilon>0$, $\gamma\in(0,1)$ such that if $U$ is a normal vector field on $\Gamma$ with $\|U\|_{C^{2,\alpha}} \leq \epsilon$, then 
\begin{align*}
|F(\Gamma_U)-F(\Gamma)|^{1+\gamma} \leq C\|\phi_U\|_{L^2}^2.
\end{align*} 
\end{theorem}

From this we derive the following differential inequality:

\begin{theorem}\label{T:differential inequality compact}
Fix $n,N$. There exist $C_1, \epsilon>0$, $\gamma\in(0,1)$ and $t_0=t_0(K)$ such that if $\Sigma_t$ is a RMCFf which is $(C^{2,\alpha}, \epsilon)$-close to some closed shrinker $\Gamma$ for $t\in[t_1,t_2]$, $t_1\geq t_0$, then
\begin{align}
\label{eq:compact-L}
\pr_t \tilde{F} \leq -C_1 (\tilde{F}-F(\Gamma))^{1+\gamma} + C_1 e^{-(1+\gamma)t}.
\end{align}
\end{theorem}
\begin{proof}

Combining \eqref{tilde F monotonicity} and Theorem \ref{thm:compact-lojasiewicz} gives for large enough $t_0$
\begin{align*}
    \pr_t \tilde{F}  \leq -  2\|\phi\|_{L^2}^2 \leq - C |F(\Sigma_t) - F(\Gamma)|^{1+\gamma}.
\end{align*}
Now by the triangle inequality $|\tilde{F}(t)-F(\Gamma)| \leq |F(\Sigma_t)- F(\Gamma)| + (\mu(t)-1)F(\Gamma)$. It follows that 
\begin{align*} 
\pr_t \tilde{F} \leq - C_\gamma( (\tilde{F}(t)-F(\Gamma))^{1+\gamma}  - F(\Gamma)^{1+\gamma} (\mu(t)-1)^{1+\gamma}).
\end{align*}
Using that $\mu(t)-1 \simeq K^2 e^{-t}$ for large $t$ gives the result. 
\end{proof}

We may solve the differential inequality as follows: 

\begin{lemma}\label{L:differential inequality compact}
Let $f:[1,\infty) \to [0,\infty)$ be a smooth, non-increasing function.
Suppose there are $\alpha>0$, $K_0>0$ and $E(t)\geq 0$ so that for $t\geq 1$ we have $f'(t) \leq - K_0 f^{1+\gamma} - E(t)$. If $E(t) \in O(t^\frac{1+\gamma}{\gamma})$, then there exists $C$ depending only on $K_0,E, \gamma, f(1)$ so that $f(t) \leq Ct^{-1/\gamma}$. 
\end{lemma}
\begin{proof}
Let $h(t) = f(t) - Ct^{-1/\gamma}$ where $C$ will be chosen later, but is large enough so that $h(1) <0$. Suppose $h$ is not strictly negative. Then there must be a first time $T>1$ at which $h(T)=0$. Then $h'(T)\geq 0$. On the other hand we have \begin{align*}
h'(T) \le-K_0 f(T)^{1+\gamma} + (C_E+C/\gamma) T^{-\frac{1+\gamma}{\gamma}} = - K_0 C^{1+\gamma} T^{-\frac{1+\gamma}{\gamma}} + (C_E+C/\gamma) T^{-\frac{1+\gamma}{\gamma}} .
\end{align*}
This is a contradiction if $C$ is chosen so large that $K_0 C^{1+\gamma} > C_E+C/\gamma$. 
\end{proof}

We may then use the solution and the monotonicity of $\tilde{F}$ to estimate the distance between time slices of a RMCFf which is close enough to $\Gamma$: 

\begin{theorem}\label{3.4}
Let $\Gamma$ be a closed shrinker. Suppose $\Sigma_t$ is a RMCFf and that we can write $\Sigma_t$ as a normal graph $U(t)$, $t\in[t_0,T]$, $t_0\geq 1$, over $\Gamma$ with $\|U(\cdot, t)\|_{C^{2,\alpha}}\le \sigma_0$, and $\tilde{F}(t)\geq F(\Gamma)$ for all $t\in[t_1,t_2]$. Then there exist constants $C_0 >0$, depending only on $\sigma_0,K$ and $\Gamma$, $\lambda_0$, and $\rho>0$ depending only on $\Gamma$ such that
\begin{align*}
\sup_{t_0\leq t_1\leq t_2 \leq T}\|U(t_2)-U(t_1)\|_{L^2}\le C_0 t_1^{-\rho}
\end{align*}
\end{theorem}

\begin{proof}
By closeness to $\Gamma$, it follows from the RMCFf equation that 
$$
\|\pr_t U\|_{L^2}\le C\|\phi + e^{-t/2}\mathbf{G}^\perp\|_{L^2}
$$
for some constant $C=C(\sigma_0)$.
Since, $\|\phi + e^{-t/2}\mathbf{G}^\perp\|_{L^2} \leq \|\phi\|_{L^2} + K\lambda_0 e^{-t/2}$, and the monotonicity for $\tilde F$, we have for any $\delta>0$
\begin{align*}
\begin{split}
 \int_{t_1}^{t_2} \|\pr_t U\|_{L^2} dt &\leq C\int_{t_1}^{t_2} \|\phi\|_{L^2} dt + Ce^{-t_1/2} \\&\leq C\left(\int_{t_1}^{t_2} \|\phi\|_{L^2}^2 t^{1+\delta}dt \right)^\frac{1}{2} \left(\int_{t_1}^{t_2} t^{-1-\delta} dt\right)^\frac{1}{2} + Ce^{-t_1/2}  
 \\\nonumber&\leq C\left(\int_{t_1}^{t_2} -(\pr_t \tilde{F}) t^{1+\delta} dt\right)^\frac{1}{2} \left(t_1^{-\delta} - t_2^{-\delta}\right)^\frac{1}{2} + Ce^{-t_1/2},
 \end{split}
\end{align*}
where we have used H\"{o}lder's inequality in the second line, and (\ref{tilde F monotonicity}) for the third. 

Now let $f(t) = \tilde{F}(t)-F(\Gamma)$ so that $\pr_t f= \pr_t \tilde{F} \leq 0$. 
Integrating by parts, we have \[\int_{t_1}^{t_2} -(\pr_t \tilde{F}) t^{1+\delta} dt =f(t_1)t_1^{1+\delta}-f(t_2)t_2^{1+\delta} +(1+\delta)\int_{t_1}^{t_2} f t^\delta dt.\] 
By the differential inequality, Theorem \ref{T:differential inequality compact} and Lemma \ref{L:differential inequality compact}, we have $f(t) \leq Ct^{-1/\gamma}$, where $C$ depends on $f(t_0)$. 
But then $\int_{t_1}^{t_2} -(\pr_t \tilde{F}) t^{1+\delta}dt \leq C t_1^{1+\delta -1/\gamma} + C\int_{t_1}^{t_2} t^{-1/\gamma +\delta} \leq C  t_1^{1+\delta -1/\gamma} . $ 
 Choosing $\delta$ so that $\rho := 1/\gamma-1-\delta>0$ completes the proof.
\end{proof}

\subsection{Extension of graph representation}

In order to apply the {\L}ojasiewicz inequality, we need to ensure that we are close to a model shrinker at all sufficiently large times. In the compact setting, we have the following lemma, which states that if we are initially close to a closed shrinker, then we remain close to it. 
We denote by $\Theta_{(0,0)}$ the Gaussian density at the spacetime point $(0,0)$, i.e.
\begin{align*}
\Theta_{(0,0)}(M_s)=\lim_{s\searrow 0}\int_{M_s}\rho_{0,s}.
\end{align*}

\begin{lemma}\label{Schulze2.2}
Let $\beta>1$ and $\Gamma$ be a shrinker.
For every $\sigma>0$ there exist $\epsilon_0>0$ and $\tau_0<0$ depending only on $\sigma,\beta,\Gamma,K$ such that if $M_s$ is a unit density K-almost Brakke flow with $\Theta_{(0,0)}(M)\ge F(\Gamma)$ and $\frac{1}{\sqrt{-s}}M_s$ is a smooth graph over $\Gamma$ of a normal vector field $U$ for $s\in [\beta\tau,\tau]$, where $\tau_0 \leq \tau <0$, 
\begin{align}
\label{eq:graph2.2}
\|U\|_{C^{2,\alpha}(\Gamma\times[\beta\tau,\tau])}\le \sigma
\end{align}
and 
\begin{align}
\label{eq:graph2.2a}
\sup_{s\in[\beta\tau,\tau]}\|U(\cdot,s)\|_{L^2(\Gamma)}\le \epsilon_0 ,
\end{align}
then $\frac{1}{\sqrt{-s}}M_s$ is the graph of an extended $U$ for $s\in [\beta\tau,\tau/\beta]$, with 
\begin{align}
\label{eq:graph2.2ext}
\|U\|_{C^{2,\alpha}(\Gamma\times[\beta\tau,\tau/\beta])}\le \sigma.
\end{align}
\end{lemma}

\begin{proof}
This is essentially Schulze's Lemma 2.2 in \cite{schulze2014uniqueness} and the proof goes through without major changes. 
For the convenience of the reader we provide a brief sketch nonetheless:

Let $M_s^\Gamma = \sqrt{-s} \Gamma$ be the unforced MCF induced by $\Gamma$. 
Assuming the result does not hold, we find a sequence of $K$-almost Brakke flows $M^k_s$ and $\tau_k \nearrow 0$, satisfying the assumption (\ref{eq:graph2.2}) with 
\begin{align}
\label{eq:contradiction}
\sup_{s\in[\beta\tau_k,\tau_k]}\|U(\cdot,s)\|_{L^2(\Gamma)}\le\frac1k,
\end{align}
but where 
 $\frac{1}{\sqrt{-s}}M_s^k$ is not a smooth graph over $\Gamma $ for $s\in[\tau_k,\tau_k/\beta]$ satisfying \eqref{eq:graph2.2ext}.
Let $\widetilde{M}^k_s$ be the parabolic rescaling of $M^k_s$ so that each is defined on $[-\beta,-1]$, i.e. $\widetilde M^k_s=|\tau_k|^2M^k_{|\tau_k|s}$. 
Then each $\widetilde{M}^k_s$ is a $|\tau_k|K$-almost Brakke flow. 
By the compactness theorem for almost Brakke flows (cf. \cite[Section 11]{white1997stratification}), and a diagonal argument, $\widetilde{M}_s^k$ converges to an unforced Brakke flow. 
It follows from (\ref{eq:contradiction}) and the monotonicity formula that the limit coincides with $M_s^\Gamma$ for $s\in(-\beta,0)$. The convergence is smooth on any compact subset of this interval by White's version of Brakke's regularity theorem (for almost Brakke flows) \cite{white2005local}, which gives the desired contradiction.
\end{proof}

\subsection{Uniqueness}

\begin{theorem}
Let $M_s^n$ be an embedded MCF with forcing in $\mathcal{U} \subset \mathbb{R}^N$. If one tangent flow at at a singular point is induced by a smooth closed shrinker $\Gamma$ with multiplicity 1, then the tangent flow at that point is unique. That is, any other tangent flow is also induced by $\Gamma$ (with multiplicity 1). 
\end{theorem}

\begin{proof}
We may assume without loss of generality that the singular point is $(0,0)$. Let $\Sigma_t$ be the corresponding RMCFf as in Section \ref{s:preliminaries}. 

By the convergence to $\Gamma$, we have:
\begin{itemize}
\item[($\dagger$)] For any $T_0 ,\epsilon, t_* >0$, there exists $t_0 > t_*$ so that $\Sigma_t$ is $(C^{2,\alpha},\epsilon)$-close to $\Gamma$ on $[t_0, t_0+T_0]$. 
\end{itemize}

Fix any $T_0>0$, and choose $\beta = e^{T_0}$. Let $\epsilon_0,\tau_0$ be as in Lemma \ref{Schulze2.2}, and $\rho,\sigma$ be as in Theorem \ref{3.4}. Let $\epsilon\in (0,\sigma)$ be such that $\epsilon F(\Gamma)\leq \epsilon_0/10$. By monotonicity of $\tilde{F}$, there exists $t_*$ so that $\tilde{F}(t) - F(\Gamma) <\epsilon$ for any $t\geq t_*$. Let $C_0$ be the constant in Theorem \ref{3.4}, which by the last inequality depends only on $\epsilon$ (and $\sigma$, but in particular not on the choice of $t_*$). Let $t_0$ be large enough so that $C_0 t_0^{-\rho} < \epsilon_0/10$. 

Suppose $\Sigma_t$ is $(C^{2,\alpha}, \epsilon)$-close to $\Gamma$ on $[t_0, t_0+T]$, with $t_0> \max(t_*,-\ln(-\tau_0))$ and $T\geq T_0$. Note that this holds with $T=T_0$ by ($\dagger$). Then by Theorem \ref{3.4}, we have 
\begin{equation*}
 \sup_{t_0\leq t_1\leq t_2\leq t_0+T}\|U(t_2)-U(t_1)\|_{L^2} \leq C_0 t_1^{-\rho}.
\end{equation*} 
In particular this does not depend on $T$ (nor $t_0$). By the triangle inequality we then have 

\begin{align*}
\|U(t_2) \|_{L^2} \leq \|U(t_0)\|_{L^2} + C_0 t_1^{-\rho} \leq \|U(t_0)\|_{C^{2,\alpha}} F(\Gamma) + C_0 t_0^{-\rho} \leq \epsilon F(\Gamma) + \epsilon_0/10 < \epsilon_0
\end{align*}
for all $t_2 \in [t_0, t_0+T]$. 

Applying Lemma \ref{Schulze2.2} on $[t_0, t_0+T]$, we may extend the graphical representation (for the rescaled flow) by $\log \beta = T_0$. That is, $\Sigma_t$ will be $(C^{2,\alpha},\sigma)$ close to $\Gamma$ for $t\in [t_0, t_0+ T+T_0]$. By induction on $T$, we conclude that $\Sigma_t$ is $(C^{2,\alpha},\sigma)$ close to $\Gamma$ for all $t\geq t_0$. 

Then $\|U(t_2)-U(t_1)\|_{L^2} \leq C_0 t_1^{-\rho} \to 0$ for all $t_2\geq t_1\geq t_0$, which implies uniqueness. 
\end{proof}

\begin{remark}
Schulze \cite{schulze2014uniqueness} in fact proves a somewhat stronger statement than the uniqueness of tangent flows. We have chosen to focus on uniqueness and present a more streamlined proof, although one could modify Schulze's proof in the analogous manner to prove a version of \cite[Theorem 0.1]{schulze2014uniqueness} for forced MCF. 
\end{remark}

\section{The extension step}
\label{sec:extension}

The goal of this section is to show that if we are very close enough to a cylinder $\Gamma$ on $B_R$, we are still pretty close to a cylinder on $B_{(1+\mu)R}$ for some fixed constant $\mu$ which is subject of Theorem \ref{thm:extension} below. 
In the next section we show that if we are pretty close to a cylinder on $B_{(1+\theta)R}$, we must in fact be very close to a (potentially different) cylinder on $B_R$.
Crucially, $\mu>\theta$ which allows us to iteratively apply this \emph{extension step} and \emph{improvement step} to obtain the scale comparison theorem \ref{thm:scale-comp}.

\subsection{Shrinker and localisation scales}\label{SS:localization scale}

We define a shrinker scale by 
\begin{align}\label{eq:shrinker}
e^{-R_T^2/2} = \int_{T-1}^{T+1} \|\phi\|_{L^2(\Sigma_t \cap B_{3e^{t/2}r_0})}^2 dt.
\end{align}
In comparison to \cite{colding2015uniqueness}, our scale differs by localising the integral to $B_{3e^{t/2}r_0}$.

In this section, we will often work on regions of the rescaled flow, and we would like these to correspond to regions inside the fixed ball $B_{r_0}$ for the original flow. To accomplish this, we will choose a \textit{localisation scale} which satisfies $R^{loc}_t \in o(e^{t/2})$. However, the localisation also introduces error terms, and to overcome these - see estimate \eqref{eq:step-2-estimate} - we make the specific choice $R^{loc}_t := 2\sqrt{t+1}$.

We also define $\lambda_0$ to be a constant such that $\sup_{x\in B_R, r>0} r^{-n} |\Sigma_t\cap B_r(x)| \leq \lambda_0$.
Several results in this section will be stated with this hypothesised area bound. In papers on unforced MCF, this hypothesis would follow from an entropy bound. In this work, the required area bound instead follows from bounds for the initial surface by Corollary \ref{cor:area-bound}, so long as $R< R^{loc}_t$.

\subsection{The extension step}

We may now state the extension step:

\begin{theorem}
\label{thm:extension}
Let $\Sigma_t$ be a RMCFf with $\sup_{x\in B_R, r>0} r^{-n} |\Sigma_t\cap B_r(x)| \leq \lambda_0$ for all $t$ and some constant $\lambda_0$. 
Given $\epsilon_2 >0$, there exist constants $\epsilon_3,t_0, R_0 ,C,\mu,C_l,C_g>0$ so that if $T\geq t_0$, $R_0\leq R< \min(R_T, R^{loc}_{T-1})$ and $B_R\cap \Sigma_t$ is given by the graph $U$ over a fixed cylinder $\Gamma$ with $\|U\|_{C^{2,\alpha}(B_R)}\leq \epsilon_3$ for $t\in [T-1/2,T+1]$, then for $t\in [T-1/2,T+1]$ we have:
\begin{enumerate}
\item $B_{(1+\mu)R}\cap \Sigma_t$ is contained in the graph of some extended $U$ with $\|U\|_{C^{2,\alpha}(B_{(1+\mu)R})}\leq \epsilon_2$;
\item $\|\phi\|^2_{L^2(B_{(1+\mu)R}\cap \Sigma_t)} \leq Ce^{-\frac{R_T^2}{2}} + C_g \lambda_0 e^{-T/2}$;
\item $|\nabla^l A| \leq C_l$ on $B_{(1+\mu )R} \cap \Sigma_t$ for each $l$. 
\end{enumerate}
\end{theorem}

To prove this result we follow the overall proof strategy in Section 5 of \cite{colding2015uniqueness}, which consists of three main steps:

\begin{itemize}
    \item Step 1: Curvature bounds on a larger time interval.
    \item Step 2: Cylindrical estimates on a larger time interval.
    \item Step 3: Cylindrical estimates on a larger scale.
\end{itemize}

Here `cylindrical estimates' means estimating closeness to the cylinder (in $C^{2,\alpha}$). 
To give an overview of these steps, we work backwards: The idea is that a uniform short-time stability for MCFf (Step 3) will translate to an increase in scale for the rescaled flow. 
This requires, on a larger time interval, both being close enough to the cylinder on the original scale (Step 1) and having curvature estimates on the extended scale (Step 2). 
For Step 1, by a monotonicity-type calculation we show that Gaussian area bounds can be pulled backward in time. These can be improved to curvature estimates using White's Brakke regularity theorem \cite{white2005local}, and higher order interior estimates for MCFf.
These curvature estimates imply that the speed of RMCFf is small, which is used in Step 2 to allow us to extend the graphical scale backwards in time.

\subsection{Pulling back density estimates in time}\label{SS:density}

As in \cite[Corollary 5.15]{colding2015uniqueness}, we show that we can pull back density estimates in time assuming the speed of the flow is small in an integral sense.
This comes at the cost of moving slightly inwards in space, but only by a fixed additive amount. 

Compared to Colding-Minicozzi, we need to additionally assume in the lemma below that $R$ is bounded from above by the localization scale.
Since $e^{-t_0/2} R^{loc}_{t_0} \to 0$, this allows us to control the error terms coming from the forcing term.

\begin{lemma}\label{L:density estimates back in time}
Given $\epsilon_2>0$, $\tau\in (0,1/2]$, $\lambda_0$, $r_1$, there exists $\mu_2>0$, $R_0$, $t_0$ such that the following holds: 

Suppose $t_0 \leq t_1<t_2 $ and $\{\Sigma_t\}_{[t_1,t_2]}$ is a RMCFf such that for some $\lambda_0$ we have $\sup_{x\in B_R, r>0} r^{-n} |\Sigma_t\cap B_r(x)| \leq \lambda_0$. Further suppose that $R+3 \leq R^{loc}_{t_0}$ and for $x_0 \in B_{R-R_0}$, 
\begin{align*}
\int_{t_1}^{t_2} dt\int_{B_{R+2 \cap \Sigma_t}} |\phi|^2 \rho \leq& \frac{\mu_2^2 e^{-(R+2)^2/4}}{R^2 (t_2-t_1+1)},\\
(4\pi \tau)^{-n/2} \int_{\Sigma_{t_2}} e^{-\frac{|x-x_0|^2}{4\tau}} \leq& 1+\frac12 \epsilon_2.
\end{align*}
Then
\begin{align*}
(4\pi \tau)^{-n/2} \int_{\Sigma_{t_1}} e^{-\frac{|x-x_0|^2}{4\tau}} \leq 1+ \epsilon_2.
\end{align*}
\end{lemma}

\begin{proof}
By the RMCFf equation we have
\begin{align*}
\partial_t\int_{\Sigma_t}f\rho=\int_{\Sigma_t}\langle D\log f,\phi\rangle f\rho-\int_{\Sigma_t}|\phi|^2f\rho + e^{-t/2} \int_{\Sigma_t} \langle \mathbf{G}^\perp, Df -f\phi\rangle \rho.
\end{align*}
Set $f(x)=\eta e^{\frac{|x|^2}4}e^{-\frac{|x-x_0|^2}{4\tau}}$ for a cutoff function $\eta$ chosen below.
We obtain
\begin{align*}
\begin{split}
\int_{\Sigma_{t_2}}\eta e^{-\frac{|x-x_0|^2}{4\tau}}-\int_{\Sigma_{t_1}}\eta e^{-\frac{|x-x_0|^2}{4\tau}}
=&\int_{t_1}^{t_2}\int_{\Sigma_t}\left(\langle D\eta,\phi\rangle e^{-\frac{|x-x_0|^2}{4\tau}}+\frac{\langle x_0,\phi\rangle}{2\tau} e^{-\frac{|x-x_0|^2}{4\tau}}\right)\\
&+\int_{t_1}^{t_2}\int_{\Sigma_t}\left(\left(1-\frac1\tau\right)\frac{\langle x,\phi\rangle}{2}\eta e^{-\frac{|x-x_0|^2}{4\tau}}-|\phi|^2e^{-\frac{|x-x_0|^2}{4\tau}}\right)\\ 
&+\int_{t_1}^{t_2}e^{-t/2} \int_{\Sigma_t} \langle \mathbf{G}^\perp, Df -f\phi\rangle \rho .
\end{split}
\end{align*}
Due to the area growth bound there exists $R_0=R_0(n)$ such that 
\begin{align*}
(4\pi\tau)^{-n/2}\int_{\Sigma_t\setminus B_{R_0\sqrt\tau}(y)}e^{-\frac{|x-x_0|^2}{4\tau}}\le \epsilon_2/100.
\end{align*}
Choose a cut-off function $\eta$ with $\eta\le1$, $|\nabla\eta|\le1$ and $\eta=1$ on $B_R$ and $\eta=0$ outside $B_{R+2}$.
Let $\tau\le1$.
\begin{align*}
\begin{split}
\int_{B_R\cap\Sigma_{t_1}}e^{-\frac{|x-x_0|^2}{4\tau}}\le&\int_{\Sigma_{t_1}}\eta e^{-\frac{|x-x_0|^2}{4\tau}}\\
\le&\int_{B_{R+2}\cap\Sigma_{t_2}}e^{-\frac{|x-x_0|^2}{4\tau}}\\
&+\int_{t_1}^{t_2}\int_{\Sigma_t\cap (B_{R+2}\setminus B_R)}\left( |\phi| e^{-\frac{|x-x_0|^2}{4\tau}}+\frac{|\langle x_0,\phi\rangle|}{2\tau} e^{-\frac{|x-x_0|^2}{4\tau}}\right)\\
&+\int_{t_1}^{t_2}\int_{B_{R+2}\cap\Sigma_t}\left(\left(\frac1\tau-1\right)\frac{|\langle x,\phi\rangle|}{2}\eta e^{-\frac{|x-x_0|^2}{4\tau}}+|\phi|^2e^{-\frac{|x-x_0|^2}{4\tau}}\right)\\
&+\int_{t_1}^{t_2}e^{-t/2} \int_{\Sigma_t} \langle \mathbf{G}^\perp, Df -f\phi\rangle \rho.
\end{split}
\end{align*}
Using the bounds for $\mathbf{G}$, we have $\left|\int_{t_1}^{t_2}e^{-t/2}dt \int_{\Sigma_t} \langle \mathbf{G}^\perp, f\phi\rangle \rho\right| \leq Ke^{-t_1/2} \int_{t_1}^{t_2} dt \int_{B_{R+2}\cap\Sigma_t} |\phi| e^{-\frac{|x-x_0|^2}{4\tau}}$.
As in the proof of Corollary 5.15 of \cite{colding2015uniqueness}, we can use Cauchy-Schwarz and the area growth bound to estimate
\begin{align*}
\int_{t_1}^{t_2}  \int_{B_{R+2} \cap \Sigma_t} |\phi| e^{-\frac{|x-x_0|^2}{4\tau}} \leq \sqrt{(4\pi \tau)^{n/2} (t_2-t_1)\lambda_0} e^\frac{(R+2)^2}{8} \left(\int_{t_1}^{t_2} dt \int_{B_{R+2}\cap\Sigma_t} |\phi|^2 \rho \right)^{1/2}.
\end{align*}
Also we have $|D\log f| \leq |D \log \eta| + {\frac{1}{\tau}}(R+2)$. Therefore 
\begin{align*}
\begin{split}
\left|\int_{t_1}^{t_2}e^{-t/2}dt \int_{\Sigma_t} \langle \mathbf{G}^\perp, Df \rangle \rho \right| \leq& {\frac{1}{\tau}}K \int_{t_1}^{t_2} e^{-t/2}dt \int_{B_{R+2}\cap\Sigma_t} (R+3) e^{-\frac{|x-x_0|^2}{4\tau}} \\
\leq& {\frac{1}{\tau}}Ke^{-t_1/2} (R+3) (4\pi \tau)^{n/2}\lambda_0.
\end{split}
\end{align*}
Again following \cite{colding2015uniqueness}, we have 
\begin{align*}
\begin{split}
&(4\pi \tau)^{-n/2} \int_{\Sigma_{t_1}} e^{-\frac{|x-x_0|^2}{4\tau}} \\\leq& (4\pi \tau)^{-n/2} \int_{\Sigma_{t_2}} e^{-\frac{|x-x_0|^2}{4\tau}} + C(\mu_2/\tau + \mu^2_2 + \mu e^{-t_0/2} ) + K \lambda_0 e^{-t_0/2}R^{loc}_{t_0} +\frac{\epsilon_2}{100}.
\end{split}
\end{align*}
Since $e^{-t_0/2} R^{loc}_{t_0} \to 0$ by our definition of the localisation scale, choosing $\mu$ {small} and $t_0$ large yields the result. 
\end{proof}

\subsection{Extending the curvature bound}\label{SS:curvature}
Having a density estimate, we apply pseudolocality for MCF to obtain curvature estimates.
Here we state the result in terms of the rescaled flow, and the rescaling contributes to the increase in scale.

\begin{proposition}\label{thm:curvature}
Given $n,\lambda_0$, there exist $\sigma$ and $\delta_2$ such that for any $\tau \in (0,1/2]$, there exists $\mu_2$, $t_0$ such that the following holds: 

Suppose $t_0 \leq t_1<t_2 $ and $\{\Sigma_t\}_{[t_1,t_2]}$ is a RMCFf such that for some $\lambda_0$ we have $\sup_{x\in B_R, r>0} r^{-n} |\Sigma_t\cap B_r(x)| \leq \lambda_0$. Further suppose that $R+3 \leq R^{loc}_{t_0}$ and for $x_0 \in B_{R-\sigma}$, 
\begin{align*}
\int_{t_1}^{t_2}  \int_{B_{R+2} \cap \Sigma_t} |\phi|^2 \rho dt\leq \frac{\mu^2_2 e^{-(R+2)^2/4}}{R^2 (t_2-t_1+1)}
\end{align*} 
and 
\begin{align*}
\sup_{B_{\sigma\sqrt{\tau}(x_0)\cap\Sigma_{t_2}}} |A|^2 \leq \delta_2/\tau.
\end{align*} 

Then for all $t\in [t_1 - \log(1-7\tau/8), t_1-\log(1-\tau)]$, we have 
\begin{align*}
\sup_{B_{\sqrt{\tau}/3}(e^{(t-t_1)/2} x_0) \cap \Sigma_t} (|A|^2+ \tau^l |\nabla^l A|^2) \leq C_l/\tau.
\end{align*} 
\end{proposition}
\begin{proof}
We proceed as in \cite{colding2015uniqueness}, applying White's version \cite{white2005local} of Brakke regularity theorem to the original flow as an almost Brakke flow in $B_{2r_0}$; note that we always apply it at a centre $y\in B_{r_0}$. 
Observe that Brakke's $\epsilon$ regularity theorem can be applied since the bound on $|A|^2$ implies a density estimate which we pull back in time via Lemma \ref{L:density estimates back in time}.
Interior estimates for MCF with forcing (see Appendix \ref{sec:interior-estimates}) give the higher derivative estimates. Note that we can do so because $R< R^{loc}_{t_0}$ so in particular $e^{-t_0/2}R \leq \sigma <r_0$ for large $t_0$. 

We remark that White's theorem is stated as a $C^{2,\alpha}$ estimate; one could also prove a $C^{l,\alpha}$ version of White's theorem following his arguments - see for instance \cite[Section 8]{edelen}, where such a result is proven in the free boundary setting. This would alleviate the need for the PDE interior estimates in Appendix \ref{sec:interior-estimates}, but we have included them as they are somewhat more concrete and may be of independent interest. 
\end{proof}

\subsection{The mean value inequality}\label{SS:meanvalue}

In this section we prove a mean value inequality for the rescaled flow, which will be required for the proof of the $\|\phi\|_{L^2}$ bound appearing in Theorem \ref{thm:extension}. It will show that $\|\phi\|_{L^2}$ can be controlled on most time slices by its average in time. 
First, we define the elliptic operators
\begin{align*}
L:=\mathcal L+\frac12 + \langle \cdot, A_{kl}\rangle A_{kl}
\end{align*}
and
\begin{align*}
\mathcal L:=\Delta-\frac12\nabla_{x^T}.
\end{align*}
Next, we record the evolution of $\phi$ under the rescaled flow (recall $\Pi$ is the normal projection):

\begin{lemma}
\label{lem:evol-phi}
If $\Sigma_t$ is a RMCFf, then we have the evolution equation
\begin{align*}
(\nabla_{\pr_t} - L) \phi = e^{-t/2} (\Lap \mathbf{G}^\perp + \langle \mathbf{G}^\perp, A_{ij}\rangle A_{ij} + \frac{1}{2} A(x^T, \mathbf{G}^T) - \frac{1}{2} \nabla^\perp_{x^T} \mathbf{G} + \frac{1}{2}\mathbf{G}^\perp).
\end{align*}
In particular, 
\begin{align*}
(\nabla_{\pr_t} - L) (\phi+ e^{-t/2}\mathbf{G}^\perp)=&\pr_t (e^{-t/2}\mathbf{G}^\perp) \\
=&-\frac{1}{2}e^{-t/2}\mathbf{G}^\perp + \Pi(DG\cdot (\phi+e^{-t/2}\mathbf{G}^\perp)) + (\pr_t\Pi)(\mathbf{G}). 
\end{align*}
\end{lemma}

The proof of Lemma \ref{lem:evol-phi} is deferred to Appendix \ref{sec:evol-phi}. 
We now proceed to prove a mean value inequality: 

\begin{lemma}\label{lemma:meanvalue}
There exists $t_0$ such that the following holds: 

Let $\Sigma_t$ be a RMCFf on $[t_1,t_2]$, let $\beta\in(0,t_2-t_1)$, $0<R\leq R^{loc}_{t_1}$, $|A|\le M$ on $\Sigma_s\cap B_{R+1}$ for all $t\in[t_1,t_2]$.
Moreover, let $t_0\le t_1\le t_2$. Then, there exists a $C=C(n,K,M)$
 such that
\begin{align*}
\max_{t\in[t_1+\beta,t_2]}\|\phi\|^2_{L^2(\Sigma_s\cap B_R)}\le (C+\beta^{-1}) \int_{t_1}^{t_2}dt \int_{\Sigma_t \cap B_{R+1}} |\phi|^2\rho + C e^{-t_1} \max_{t\in [t_1,t_2]} F(\Sigma_t) .
\end{align*}
\end{lemma}

Note that for our mean value inequality we follow the approach of \cite[Proof of Theorem 7.4]{colding2019regularity} instead of \cite[Lemma 5.32]{colding2015uniqueness}.

\begin{proof}
Let $0\leq \eta\leq 1$ be a cutoff function supported on $B_{R+1}$ with $\eta=1$ on $B_R$ and $|D \eta|\le2$. For the purposes of this proof define $\tilde{\phi} := (\phi + e^{-t/2} \mathbf{G}^\perp)|_{\Sigma_t}$ and note that this is the velocity of the flow. 

Consider $g(t)=\int_{\Sigma_t}|\tilde{\phi}|^2\eta^2e^{\frac{-|x|^2}4}$.

Arguing similarly to \cite{colding2015uniqueness} we have 
\begin{equation}
\begin{split}
\label{eq:meanvalue-1}
 g'(t) &= \int_{\Sigma_t\cap B_{R+1}} \rho\left(\eta^2 (\pr_t - \mathcal{L})|\tilde{\phi}|^2 + \eta^2 \mathcal{L}|\tilde{\phi}|^2 + |\tilde{\phi}|^2 \langle D\eta^2 , \tilde{\phi}\rangle -\eta^2 \langle \phi,\tilde{\phi}\rangle |\tilde{\phi}|^2 \right)
\\& = \int_{\Sigma_t\cap B_{R+1}} \rho\left(\eta^2 (\pr_t - \mathcal{L})|\tilde{\phi}|^2 -\langle \nabla \eta^2,\nabla |\tilde{\phi}|^2\rangle + |\tilde{\phi}|^2 \langle D\eta^2 , \tilde{\phi}\rangle -\eta^2 \langle \phi,\tilde{\phi}\rangle |\tilde{\phi}|^2 \right).
\end{split}
\end{equation}

We estimate each of these four terms individually. First, since $(\pr_t-L)\tilde{\phi}=0$ we have 
\begin{equation}
\label{eq:meanvalue-2}
(\pr_t -\mathcal{L})|\tilde{\phi}|^2 = 2\langle \tilde{\phi}, (\pr_t-\mathcal{L})\tilde{\phi}\rangle - 2|\nabla \tilde{\phi}|^2 = (2|A|^2+1)|\tilde{\phi}|^2 - 2|\nabla \tilde{\phi}|^2.
\end{equation}
For the second we estimate 
\begin{equation}
\label{eq:meanvalue-3}
4\eta |D\eta| |\tilde{\phi}| |\nabla \tilde{\phi}| \leq \eta^2 |\nabla \tilde{\phi}|^2 + 4|\tilde{\phi}|^2 |D\eta|^2.\end{equation} 
For the third, we use \begin{equation}
\label{eq:meanvalue-4}
2|\tilde{\phi}|^3 \eta |D\eta| \leq \frac{1}{2}\eta^2|\tilde{\phi}|^2  +2|\tilde{\phi}|^2 |D\eta|^2.\end{equation}
Finally, we have 
\begin{equation}
\label{eq:meanvalue-5}
-\eta^2 \langle \phi,\tilde{\phi}\rangle|\tilde{\phi}|^2 = -\eta^2 |\tilde{\phi}|^4 + e^{-t/2} \eta^2 \langle \mathbf{G}^\perp, \tilde{\phi}\rangle |\tilde{\phi}|^2 \leq -\eta^2 |\tilde{\phi}|^4 + \frac{1}{2}Ke^{-t/2} (|\tilde{\phi}|^2 + |\tilde{\phi}|^4). 
\end{equation}
Combining (\ref{eq:meanvalue-1}-\ref{eq:meanvalue-4}) and using $|D\eta|\leq 2$, $|A|\leq M$ then gives 
\begin{equation*}
\begin{split}
g'(t) &\leq  \int_{\Sigma_t} \rho\left(\frac{1}{2}(Ke^{-t/2}-1)\eta^2 |\tilde{\phi}|^4 -|\nabla\tilde{\phi}|^2 +( 2M+1 + \frac{1}{2}Ke^{-t/2}) \eta^2|\tilde{\phi}|^2 + 6|D\eta|^2 |\tilde{\phi}|^2 \right)
\\& \leq C(M,K) \int_{\Sigma_t\cap B_{R+1}} \eta^2|\tilde{\phi}|^2\rho,
\end{split}
\end{equation*}
as long as $t_0$ is so large that $Ke^{-t_0/2} \leq 1$. Take $t_*\in[t_1,t_1+\beta]$ so that $g(t_*) = \min_{t\in [t_1,t_1+\beta]} g(t)$. Then for $t\in[t_1+\beta,t_2]$ we have
\begin{equation*}
\begin{split}
g(t) &= g(t_*) + \int_{t_*}^{t} g'(\tau) d\tau \leq  \frac{1}{\beta}\int_{t_1}^{t_1+\beta} g(\tau) d\tau + C \int_{t_1}^{t} d\tau \int_{\Sigma_\tau\cap B_{R+1}} \eta^2|\tilde{\phi}|^2\rho
\\& \leq (C+\beta^{-1}) \int_{t_1}^{t} d\tau \int_{\Sigma_\tau\cap B_{R+1}} \eta^2|\tilde{\phi}|^2\rho. 
\end{split}
\end{equation*}
By the squared triangle inequality we now have 
\begin{equation*}
\begin{split}
\max_{t\in[t_1+\beta,t_2]}\|\phi\|^2_{L^2(\Sigma_s\cap B_R)} &- K^2 e^{-t} \max_{t\in [t_1+\beta,t_2]} F(\Sigma_t)^2 \leq 2 \max_{t\in[t_1+\beta,t_2]}\|\tilde{\phi}\|^2_{L^2(\Sigma_t\cap B_R)} 
\\& \leq 4(C+\beta^{-1}) \int_{t_1}^{t_2} dt \int_{\Sigma_t\cap B_{R+1}} \eta^2|\tilde{\phi}|^2\rho + 4(C+\beta^{-1}) K^2\int_{t_1}^{t_2} e^{-t}F(\Sigma_t)^2dt,
\end{split}
\end{equation*}
which implies the result. 
\end{proof}

Note that by the almost monotonicity (Section \ref{sec:almost-monotonicity}), we can assume $F(\Sigma_t) \leq \lambda_0$ for $t\geq t_0$.

\subsection{Short time stability of the cylinder}\label{SS:cylinder}

We need a short-time stability result for solutions of forced MCF. Specifically, we consider MCF with forcing as a parabolic system on the normal bundle over a base submanifold $\Sigma$. Intuitively, if the velocity of the flow is bounded (in $C^{2,\alpha}$, say) then the solutions must stay close to $\Sigma$ in a parabolic neighbourhood.
Even though the proof below of this lemma is elementary it is one of the main ingredients to extend our graphical scale in Theorem \ref{thm:extension}.
Moreover, it is the only `parabolic ingredient' of the proof.

\begin{lemma}
\label{lem:unif-stability}
Let $\Sigma \subset \mathbb{R}^N$ be a complete submanifold with uniformly bounded geometry, so that $\sup_\Sigma \sum_{j=0}^3 |\nabla^j A^\Sigma| <\infty$. There exists $R_0$ such that for every $R>R_0$, $\epsilon>0$ and $C_0>0$, there are $\delta_3>0$ and $\gamma>0$ such that if $M_t$ is a MCF with forcing term $\mathbf{F}$ satisfying
\begin{itemize}
\item $B_{R+2}\cap M_{-1}$ is a normal graph $U$ over $\Sigma\in \mathcal C_k$ with $\|U\|_{C^{2,\alpha}}\le\delta_3$;
\item $|A|+|\nabla A|+|\nabla^2A|+|\nabla^3A|+|\mathbf{F}|+|D \mathbf{F}|+|D^2 \mathbf{F}|+|D^3\mathbf{F}|\le C_0$ on $B_{R+2}\cap M_t$ for $t\in[-1-\frac1{C_0},-1+\frac1{C_0}]$;
\end{itemize}
then for each $t\in[-1-\gamma,-1+\gamma]$, we have that $B_R\cap M_t$ is a normal graph over $\sqrt{-t}\Sigma$ with $C^{2,\alpha}$ norm at most $\epsilon$.
\end{lemma}

\begin{proof}
Since $|A|$ and $|\mathbf{F}|$ are bounded, the MCF with forcing equation implies that $|\partial_tx|$ is also bounded. 
Likewise, the bounds on $|D \mathbf{F}|$ and $|\nabla A|$ (and thus on $|\nabla \mathbf{H}|$) implies that also $|\partial_t\Pi|$ is uniformly bounded, where $\Pi$ is the projection onto the normal bundle.
Combining these bounds, it follows that $B_{R+1}\cap M_t$ remains graph over $\Sigma$ of a normal vector field $U$ with uniform bounds
\begin{align*}
|\partial_t U|+|\partial_t\nabla U|\le C_1\quad\text{for $t\in[-1-\theta_2,-1+\theta_2]$},
\end{align*}
where $\theta_2>0$ and $C_1$ depends on $C_0,\epsilon,n$.
The higher order bounds follow in the same fashion.
\end{proof}

\subsection{Proof of the extension step}\label{SS:proof}

We proceed with the main theorem of this section.
As mentioned at the beginning of the section, we first establish curvature estimates backward in time, then show that the speed of the flow is small which then allows us to extend the cylindrical scale by the short time stability of the cylinder.

In our presentation of this subsection, we adopt some clarifications based, in part, on notes of Mantoulidis \cite{Mantoulidis}, which the reader may also find helpful.

\begin{proof}[Proof of Theorem \ref{thm:extension}]

\textbf{Step 1: Curvature bounds on a larger time interval.} First, we establish some curvature bounds on some extended scale backwards in time.
Let $A_0$ be the curvature of the cylinder $\Gamma$. 
We choose $\epsilon_3$ so that the curvature of any surface $\Gamma'$ which is $(C^{2,\alpha},\epsilon_3)$-close to $\Gamma$ has curvature at most $A_0+1$. 
Let $\delta_2$ be the constant from Proposition \ref{thm:curvature}.
Take a constant $\tau \leq \frac{1}{100}$ which is small enough that $(A_0+1)^2\leq\frac{\delta_2}\tau$. 
Then in particular, $|A|^2 \leq \delta_2 /\tau$ for each point in $\Sigma_t \cap B_R$, $t\in [T-1/2,T+1]$. 

Now using that $R+2\le 3r_0 e^{T/2}$ and $R\le R_T$, we have by definition of shrinker scale
\begin{align*}
\int_{T-1}^{T+1} dt \int_{\Sigma_t \cap B_{R+2}} |\phi|^2 \rho \leq  \int_{T-1}^{T+1} \|\phi\|_{L^2(\Sigma_t \cap B_{3r_0 e^{t/2}})}^2 dt  =  e^{-R_T^2/2}.
\end{align*}

Moreover, by choosing $t_0$ and $R$ sufficiently large we can ensure that
\begin{align*}
e^{-R_T^2/2} < \frac{\mu_2^2 e^{-(R+2)^2/4}}{2R^2}
\end{align*}
where $\mu_2$ is the small constant from Proposition \ref{thm:curvature}.

Let $\sigma$ be as in the statement of Proposition \ref{thm:curvature}. We apply that proposition at all $x \in B_{R-\sigma}$ with $t_2=T+1$. This will give curvature estimates at times $t\in [t_1 - \log(1-7\tau/8), t_1-\log(1-\tau)]$, for any $t_1 \leq T+1-\tau$. In particular, we conclude that for each $l$, 
\begin{align*}
\sup_{B_{R_1} \cap \Sigma_t} (|A|^2+ \tau^l |\nabla^l A|^2) \leq C_l/\tau,
\end{align*} 
for any $t\in [T-1-\log(1-7\tau/8), T+1-\tau + \log(1-\tau)]$, where $R_1= \frac{\sqrt\tau}3+(1-\tau)^{-\frac12}(R-\sigma)$. Choosing $R_0\gg1$ sufficiently large, we have $R_1\ge (1+\kappa)R$ for some $\kappa \in (0,\sqrt{2}-1)$.
Hence we have obtained a curvature estimate on a larger time interval $[T-3/4,T+1+\gamma]$ and a larger ball $B_{(1+\kappa)R}$. Here $\gamma=\gamma(\tau)>0$. 

\

\textbf{Step 2: Cylindrical estimates on a larger time interval.} Having established curvature bounds, by the mean value inequality (Lemma \ref{lemma:meanvalue}) 
and the definition of shrinker scale (\ref{eq:shrinker}), for $t\in [T-7/8,T+1]$ we now have
\begin{align*}
\|\phi\|_{L^2(B_{(1+\kappa) R}\cap \Sigma_t )}^2 \leq C\int_{T-1}^{T+1} \|\phi\|_{L^2(\Sigma_t \cap B_{3e^{t/2}r_0})}^2 dt + C_g \lambda_0 e^{-T/2} =Ce^{-R_T^2/2} + C_g \lambda_0 e^{-T/2}.
\end{align*}

By interpolation (cf. \cite[Appendix B]{colding2015uniqueness} or \cite[Appendix]{zhu2020ojasiewicz}), the curvature bounds, and since $R< R_T$, for $t$ in the time interval $[T-3/4,T+1]$, we have
\begin{align}
\label{eq:step-2-estimate}
\begin{split}
\|\phi\|_{C^{2,\alpha}(B_{(1+\kappa)R-1}\cap \Sigma_t)} \leq &C (e^{(1+\kappa)^2 R^2/8} \|\phi\|_{L^2(B_{(1+\kappa)R}\cap \Sigma_t)})^{1-\delta_l} \\
\leq& C(Ce^{(1+\kappa)^2 R^2/8}e^{-R_T^2/4} + e^{(1+\kappa)^2R^2/8} e^{-T/2})^{1-\delta_l}\leq C.
\end{split}
\end{align}
Note that we have used $R<R^{loc}_t$ to control the term $e^{(1+\kappa)^2R^2/8} e^{-T/2}$.

 Since $\nabla^lA$ is bounded and $\mathbf{G}$ is bounded, $\mathbf{G}^\perp$ is also bounded in $C^{2,\alpha}(B_{R})$ for $t\in [T-3/4,T+1]$. 
Hence the velocity $\phi+e^{-\frac t2}\mathbf{G}^\perp$ of the rescaled flow is bounded.
Thus, the initial $(C^{2,\alpha},\epsilon_3)$-closeness to $\Gamma$ on $ B_R$, for $t\in [T-1/2,T+1$], extends to give $(C^{2,\alpha},2\epsilon_3)$-closeness on $B_{R-1}$, for $t\in [T-1/2-\xi,T+1]$. 
Here $\xi>0$ depends only on $\epsilon_3$, $\|\mathbf{G}^\perp\|_{C^{2,\alpha}}$ and $\|\phi\|_{C^{2,\alpha}}$. 

\textbf{Step 3: Cylindrical estimates on a larger scale.} We have now established curvature bounds on $\Sigma_t\cap B_{(1+\kappa)R}$, $t\in [T-1/2-\xi, T+1+\gamma]$, and cylindrical estimates on $B_R$ for $t\in [T-1/2-\xi, T+1]$. Take $C_0 > \max(1/\xi,1/\gamma)$ and let $\mu$ be as in the short-time stability Lemma \ref{lem:unif-stability}. We also may choose $\delta<\delta_3/2$. 

Then for any fixed $t\in [T-1/2 - \theta, T+1]$, we may apply that lemma to the MCFf starting from $\Sigma_{t}$; the conclusion at time $t+\mu$ (translated to RMCFf) implies that $\Sigma_{t+\mu}$ is $(C^{2,\alpha},2\epsilon_2)$-close to $\Gamma$ on $B_{(1+\mu)R}$. In particular, this establishes the desired cylindrical estimates on $B_{(1+\mu)R}$ for any $t\in [T-1/2,T+1]$. 
\end{proof}


\section{Shrinker scale and cylindrical scale}
\label{sec:scale-comparison}

\subsection{Improvement step}

The following {\L}ojasiewicz inequality follows from the work of Colding-Minicozzi. Note that the {\L}ojasiewicz inequality is purely a statement about submanifolds, and does not explicitly involve any flow. 

\begin{theorem}[\cite{colding2019regularity}]
There exists $\epsilon_2>0$ such that given $\epsilon_1>0$, $\lambda_0$, $\gamma>0$ and $\beta, \bar{\beta},\kappa <1$, there exist $R_0, l>0$ and $C_{\beta,\bar{\beta},\kappa}$ such that if $\Sigma^n\subset \mathbb{R}^N$ has $\lambda(\Sigma)\leq \lambda_0$ and:
\begin{enumerate}
\item For some $R>R_0$, we have that $B_R\cap \Sigma$ is a $C^{2,\alpha}$ normal graph $U$ over some cylinder with $\|U\|_{C^{2,\alpha}(B_R)} \leq \epsilon_2$;
\item $|\nabla^j A| \leq C_j$ on $B_R\cap \Sigma$ for all $j \leq l$;
\end{enumerate}
then $B_{(1-\gamma)R_1}\cap \Sigma$ is a graph $V$ over some (possibly different) cylinder with $\|V\|_{C^{2,\alpha}} \leq \epsilon_1$ and $\|V||_{L^2(B_{(1-\gamma)R_1})} \leq e^{-(1-\gamma)^2\frac{R_1^2}{4}} $, where 
\begin{align}\label{5.1}
R_1 = \max \left\{ r\leq R-1 \,\Big| \,  C_{\beta,\bar{\beta},\kappa} \left( \|\phi_U\|_{L^2}^\frac{6\bar{\beta}}{3+\kappa} + \|\phi_U\|_{L^1}^{\bar{\beta}} + \|U\|_{L^2}^\frac{6}{3-\kappa} + R^{n-2} e^{-R^2/4} \right)\leq e^{-r^2/4} \right\}.
\end{align}

\end{theorem}

\begin{proof}
This essentially follows from the proof of (2) in the proof of \cite[Theorem 7.4]{colding2019regularity}; the point is to apply \cite[Proposition 4.47]{colding2019regularity} on the largest scale possible $R_1$. The only difference is that we do not assume explicit bounds on $\phi$, so the statement includes the interpolated terms $\|\phi_U\|_{W^{1,2}} \leq \|\phi_U\|_{L^2}^{\bar{\beta}}$ and $\|\phi_U\|_{W^{2,1}} \leq \|\phi_U\|_{L^1}^{\bar{\beta}}$.
More precisely, to apply Proposition 4.47 in \cite{colding2019regularity}, we need to have bounds on $|\phi|_2$ and $|\nabla\tau|^2_1$ which is shown in the proof of Theorem 7.4.
In particular, in our case we obtain the bounds from equation (7.23) without further estimating $\phi$. 
\end{proof}

The above may be used to prove the following scale improvement theorem:

\begin{theorem}[\cite{colding2019regularity}]
\label{thm:improvement}
There exists $\epsilon_2>0$ such that given $\epsilon_1>0$, $\mu>0$, $\lambda_0$, there exist $R_0, l_0>0$ and $\theta\in(0,\mu)$ such that if $\Sigma^n\subset \mathbb{R}^N$ has $\lambda(\Sigma)\leq \lambda_0$ and:
\begin{enumerate}
\item For some $R_0\leq R\leq R_*$, we have that $B_R\cap \Sigma$ is a $C^{2,\alpha}$ graph $U$ over some cylinder with $\|U\|_{C^{2,\alpha}(B_R)} \leq \epsilon_2$ and $\|U\|_{L^2(B_R)}^2\leq C_n \lambda_0 R^{n-2} e^{-\frac{R^2}{4(1+\mu)^2}}$; 
\item $\|\phi\|_{L^2(B_R\cap\Sigma)}^2 \leq C_2 e^{-R_*^2/2}$; 
\item $|\nabla^l A| \leq C_l$ on $B_R\cap \Sigma$ for all $l\leq l_0$;
\end{enumerate}
then $B_{R/(1+\theta)}\cap \Sigma$ is a graph $V$ over some (possibly different) cylinder with $\|V\|_{C^{2,\alpha}} \leq \epsilon_1$ and $\|V||_{L^2(B_{R/(1+\theta)})} \leq e^{-\frac{R^2}{4(1+\theta)^2}} $. 
\end{theorem}

\begin{proof}

We want to use the assumed estimates for $U$ and $\phi$ so that when we apply the {\L}ojasiewicz inequality, we will have $R_1\geq \frac{R}{(1-\gamma)(1+\theta)}$ for some $\gamma>0$. Thus the goal is to show that for some $\kappa\in(0,1]$, $\beta<1$ and large enough $R$ we will have
\begin{equation}
\label{eq:improvement-1}
C( \|\phi_U\|_{L^2}^\frac{6\bar{\beta}}{3+\kappa} + \|\phi_U\|_{L^1}^{\bar{\beta}} + \|U\|_{L^2}^\frac{6}{3-\kappa} + R^{n-2} e^{-R^2/4}) \leq e^{-\frac{ R^2}{4\beta(1+\theta)^2}} .
\end{equation}

To do so, we first note that $\|\phi_U\|_{L^2}$ and $\|\phi\|_{L^2}$ differ essentially by the $L^2$ norm of $\phi$ outside of $B_R$. The resulting error term may be estimated using \cite[Lemma 7.16]{colding2019regularity}, and we will refer to similar error terms as \textit{cutoff error}. On the other hand, $\|\phi\|_{L^2}$ is bounded by assumption (2).

In fact, since $R<R_*$, cutoff error contributes the dominant term, and we have $\|\phi_U\|_{L^2}^2 \leq C R^n e^{-R^2/4}$ and $\|\phi_U\|_{L^1} \leq C R^{n-1} e^{-R^2/4}$. 

This bounds the left hand side of (\ref{eq:improvement-1}) by a constant times $\|U\|_{L^2}^\frac{6}{3-\kappa} + R^n e^{-\frac{3\bar{\beta}}{3+\kappa} \frac{R^2}{4}}$. The latter term is dominated by $e^{-\frac{ R^2}{4\beta(1+\theta)^2}}$ so long as $\frac{3\bar{\beta}}{3+\kappa} > \frac{1}{\beta(1+\theta)^2}$. Finally we need $\frac{3}{(3-\kappa)(1+\mu)^2} > \frac{1}{\beta(1+\theta)^2}$. We can always choose $\beta,\kappa$ so that $\frac{3\beta}{3-\kappa} > \frac{(1+\mu)^2}{(1+\theta)^2}$ after choosing $\theta$ very close to $\mu$, which completes the proof. 
\end{proof}

\begin{remark}
The scale improvement Theorem \ref{thm:improvement} in fact holds for certain generalised shrinking cylinders by work of the second named author; see \cite[Theorem 7.2]{zhu2020ojasiewicz}. 
\end{remark}

\subsection{Scale comparison}
\label{sec:scale-comp}

We define now $R_*$ by $e^{-R_*^2/2} = e^{-R_T^2/2} + e^{-T/2}$. Combining our extension step above with the Colding-Minicozzi improvement step shows that the graphical scale extends to a fixed factor larger than $R_*$ by a bootstrapping argument.
In other words, we can precisely control the size of $\Sigma_t$ which is close to a cylinder and the rate is given by $R_\ast$.

\begin{theorem}[Scale comparison]
\label{thm:scale-comp}
Given $\epsilon_0>0$, there exist $R_1, \mu>0$ and $\epsilon_1>0$ such that if $\Sigma_t$ is a RMCFf that is $C^{2,\alpha}, \epsilon_1$ close to a fixed cylinder on $B_{R_1}$ for $t\in [T-1,T+1]$, then there are $C_l$ such that for each $t\in [T-1/2, T+1]$:
\begin{enumerate}
\item $B_{(1+\mu)R_*} \cap \Sigma_t$ is a graph of some $U$ with $\|U\|_{C^{2,\alpha}} \leq \epsilon_0$, $\|U\|_{L^2}^2 \leq C_n\lambda_0 R_T^{n-2} e^{-R_*^2/4}$ and $\|\phi_U\|_{L^2}^2 \leq e^{-(1+\mu)^2\frac{R_*^2}{4}}$;
\item For each $l$ we have $\sup_{B_{(1+\mu)R_*} \cap \Sigma_t} |\nabla^l A| \leq C_l$. 
\end{enumerate}
\end{theorem}
\begin{proof}
The point is to use the extension step, Theorem \ref{thm:extension}, to extend the graphical scale by factor $1+\mu$ and the improvement step, Theorem \ref{thm:improvement}, to retain good estimates, after coming in by a factor $1+\theta$, where $\theta<\mu$.  
Take $\epsilon_2$ as in Theorem \ref{thm:improvement} and $\epsilon_3$ as given by Theorem \ref{thm:extension}. 
We may do so under the inductive hypotheses that for each $t\in [T-1/2,T+1]$:
\begin{enumerate}
\item $B_R \cap \Sigma_t$ is given by the graph of $U$ over a fixed cylinder $\Gamma$ with $\|U\|_{C^{2.,\alpha}(B_R)} \leq \epsilon_3$;
\item For each $l$ we have $\sup_{B_{R} \cap \Sigma_t} |\nabla^l A| \leq C_l$.
\item $\|U\|_{L^2(B_R)}^2 \leq e^{-R^2/4}$;
\end{enumerate} 
For $\epsilon_1\leq \epsilon_3$ small enough depending on $R_1$, these will be satisfied at the initial scale $R_1$. Assuming the inductive hypothesis at scale $R \leq R_*$, we may apply Theorem \ref{thm:extension} to extend the graph $U$ and the curvature estimates to scale $(1+\mu)R$. By hypothesis (3), the extended graph will have $L^2$ norm dominated by cutoff error, that is, by \cite[Lemma 7.16]{colding2019regularity},
\begin{equation}
\label{eq:scale-ext-L2}
\|U\|_{L^2(B_{(1+\mu)R})}^2 \leq C_n\lambda_0 R^{n-2} e^{-R^2/4}.
\end{equation} 

By conclusion (2) of the extension step and the definition of $R_*$, we have \[\|\phi\|^2_{L^2(B_{(1+\mu)R}\cap \Sigma_t)} \leq C_2 e^{-R_*^2/2},\] so long as $(1+\mu)R\leq R_*$. We may then apply Theorem \ref{thm:improvement} (on each time-slice). The conclusions of the improvement step mean that the inductive hypothesis is satisfied at scale $\frac{1+\mu}{1+\theta}R$. 

At the last iteration, we use the extension step one last time to extend the scale to $(1+\mu)R_*$. Equation (\ref{eq:scale-ext-L2}) at this scale gives desired the $L^2$ estimate for $U$. The $C^2$ bounds for $U$ follow by interpolation. Finally, the $L^2$ norm for $\phi_U$ at this scale is also dominated by cutoff error; by \cite[Lemma 7.16]{colding2019regularity} again we have 
\begin{align*}
 \|\phi_U\|_{L^2(B_{(1+\mu)R_*})}^2 \leq e^{-R_*^2/2} + C R_*^n  e^{-(1+\mu)^2R_*^2/4}.
\end{align*}
 Taking $\mu'$ slightly smaller than $\mu$, we can always assume enough initial closeness and that $t_0$ is large enough so that $R_*$ is also large and in particular $\|\phi_U\|_{L^2(B_{(1+\mu)R_*})}^2$ is bounded above by $e^{-(1+\mu')^2R_*^2/4}$. 
\end{proof}

\section{Uniqueness of cylindrical tangent flows}
\label{sec:uniqueness-cyl}

In this section, we prove uniqueness for the cylindrical case, Theorem \ref{thm:main}. The overall structure is similar to Section \ref{sec:compact}, but there are several modifications to handle the noncompactness, compounded by only having almost-monotonicity for $F$. Throughout this section, we consider a RMCFf $\Sigma_t$ as in Section \ref{sec:localisation-rescaling}, with $\lambda(M_{s_*}) \leq \lambda_0$. 

\subsection{Almost monotonicity controls $\phi$}

Here we again show that a modified functional $\tilde{F}$ is monotone and its gradient controls the shrinker quantity $\phi$, although the definition is more complicated than the compact case because of the need for localisation. 

Recall the notation of Section \ref{sec:localisation-rescaling}, Section \ref{sec:almost-monotonicity}, and in particular the fixed cutoff function $\psi$ supported on $B_{4r_0}$. Let $\psi_t(x) = \psi(e^{-t/2}x)$, 
\begin{equation*}
\hat{F}(t) = F^{\psi_t^2}_{0,1}(\Sigma_t) = \int_{\Sigma_t} \psi_t^2 \rho,
\end{equation*}
and $\mu(t) = \exp(K_1 e^{-t})$, where $K_1 = K^2 + 2K_\psi^2 r_0^2 + KK_\psi r_0^{-1}$. The modified functional we consider is 
\begin{equation*}
\tilde{F}(t) := \mu(t) \hat{F}(t) + K_3 e^{-nt/2}, 
\end{equation*}
where $K_3 = \frac{4K_2}{n}$, $K_2 = 4K_\psi C_n \lambda_0 (12\pi r_0)^{n/2}$, and $C_n$ is a constant depending only on $n$ which will be determined below.

\begin{lemma}
Assume $\Sigma_t$ is an RMCFf as above. Then there exists a $t_0 \geq0$ such that for all $t_1,t_2\ge t_0$ we have
\begin{equation}
\label{eq:phi-Ftilde}
\int_{t_1}^{t_2} dt \int_{\Sigma_t \cap B_{3e^{t/2}r_0}} |\phi|^2 \rho  \leq 2(\tilde{F}(t_1)-\tilde{F}(t_2)).\end{equation}
\end{lemma}

\begin{proof}
We compute
\begin{align*}
\frac{d}{dt} \hat{F} = -\int_{\Sigma_t} \psi_t^2 \langle \phi, \phi+ e^{-t/2}\mathbf{G} \rangle\rho  + \int_{\Sigma_t} 2\psi_t \langle D\psi (e^{-t/2}x), e^{-t/2} \phi + e^{-t}\mathbf{G} - e^{-t/2}x/2\rangle \rho .
\end{align*}
We estimate $e^{-t/2} |\phi||\mathbf{G}| \leq \frac{1}{4}|\phi|^2 + e^{-t}K^2$ and $2e^{-t/2} \psi_t |D\psi| |\phi| \leq \frac{1}{4}|\phi|^2 \psi_t^2 + 2e^{-t} |D\psi|^2$. We also have $|D\psi| \leq K_\psi r_0^{-1}$, and for $e^{-t/2}x$ to be in the support of $|D\psi|$ we must have $ 3r_0 \leq e^{-t/2}|x| \leq 4r_0$. 
This yields 
\begin{align*}
\begin{split}
\frac{d}{dt} \hat{F}  \leq &-\frac{1}{2}\int_{\Sigma_t} \psi_t^2 |\phi|^2\rho + (e^{-t} K^2 + 2e^{-t} K_\psi^2 r_0^{-2} + e^{-t}KK_\psi r_0^{-1}) \hat{F}(t)  \\
&+ 4K_\psi\int_{\Sigma_t} \rho \mathbf{1}_{B_{4e^{t/2}r_0} \setminus B_{3e^{t/2}r_0}}.
\end{split}
\end{align*}
Using the area growth bound (Corollary \ref{cor:area-bound}) on the last term we have
\begin{align*}
\frac{d}{dt} \hat{F}  \leq -\frac{1}{2} \int_{\Sigma_t}\psi_t^2 |\phi|^2\rho + e^{-t} K_1 \hat{F}(t)  + K_2 e^{-nt/2}, 
\end{align*} 
where $K_1 = K^2 + 2K_\psi^2 r_0^2 + KK_\psi r_0^{-1}$ and $K_2 = 4K_\psi C_n \lambda_0 (12\pi r_0)^{n/2}$. 
Choose $t_0$ so that $\mu(t_0)=2$. Then for $t\geq t_0$, we have $1\leq \mu\leq 2$, so since
\begin{align*}
\tilde{F}(t) = \mu(t)\hat{F}(t)  + K_3 e^{-nt/2}
\end{align*}
where $K_3 = \frac{4K_2}{n}$, we have
\begin{align*}
\frac{d}{dt} \tilde{F} \leq -\frac{1}{2}\mu(t) \int_{\Sigma_t} \psi_t^2 |\phi|^2\rho \leq -\frac{1}{2} \int_{\Sigma_t} \psi_t^2 |\phi|^2\rho .
\end{align*}
In particular, $\tilde{F}(\Sigma_t)$ is non-increasing, and integrating gives
\begin{equation}
\label{eq:phi-tilde-F}
\int_{t_1}^{t_2}dt \int_{\Sigma_t \cap B_{3e^{t/2}r_0}} |\phi|^2 \rho  \leq 2(\tilde{F}(t_1)-\tilde{F}(t_2))
\end{equation}
which finishes the proof.
\end{proof}

\subsection{Discrete differential inequality}

The required {\L}ojasiewicz inequalities follow from the improvement and extension steps above. We also need the following fact:

By \cite[Proposition 6.5]{colding2019regularity}, for compactly supported normal graphs over the cylinder with graph function $U$ with small enough $C^2$ norm, we have \begin{equation} |F(\Gamma_U) - F(\Gamma)| \leq C\|\phi_U\|_{L^2} \|U\|_{L^2} + C\|U\|_{L^2}^3.\end{equation}

We now proceed to prove the discrete differential inequality for the cylindrical case: 

\begin{theorem}\label{T:discrete}
Given $n,N$, there exist $K,C, \epsilon, t_0$ and $\mu\in(0,1/2)$ such that if $\Sigma_t$ is a RMCFf which is $(\epsilon,R_1, C^{2,\alpha})$-close to some cylinder $\Gamma$ for $t\in[T-1,T+1]$, $T\geq t_0$, then 
\begin{align*}
|\tilde{F}(T)-F(\Gamma)| \leq K(\tilde{F}(T-1)- \tilde{F}(T+1)^{\frac{1+\mu}{2}} + C e^{-\frac{1+\mu}{4}T}.
\end{align*}
\end{theorem}
\begin{proof}
First, by the triangle inequality we have 
\begin{align*}
|\tilde{F}(t) - F(\Gamma)|\leq K_3 e^{-nt/2} + \mu(t)|\hat{F}(t)-F(\Gamma)| + F(\Gamma) (\mu(t)-1).
\end{align*} 

Recall the definition of $R_*$ from Section \ref{sec:scale-comp}; by the scale comparison Theorem \ref{thm:scale-comp}, we know that $\Sigma_T$ is a graph $U$ over the cylinder $\Gamma$ at scale $B_{(1+\mu)R_*}$. So again using \cite[Lemma 7.16]{colding2019regularity} to estimate the cutoff error, we have 
\begin{align*}
|\hat{F}(T) - F(\Gamma)| = |F^{\psi_T^2}_{0,1}(\Sigma_T)-F(\Gamma)| \leq |F^{\psi_T^2}_{0,1}(\Gamma_U)-F(\Gamma)| + C_n\lambda_0 R_*^{n-2} e^{-(1+\mu)^2R_*^2/4}.
\end{align*}

Since $R_* \leq R^{loc}_{T-1}$, we have $F^{\psi_T^2}_{0,1}(\Gamma_U) = F(\Gamma_U)$. 

Moreover the scale comparison shows that $\|U\|_{L^2}^2 \leq C_n\lambda_0 R_*^{n-2} e^{-R_*^2/4}$ and $\|\phi_U\|_{L^2}^2 \leq e^{-(1+\mu)^2\frac{R_*^2}{4}}$, so by the estimate above we have 
\begin{align*}
|F(\Gamma_U)-F(\Gamma)| \leq C R_*^{n-2} e^{-\frac{(1+(1+\mu)^2)}{8}R_*^2}.
\end{align*}

Taking $\mu'$ slightly smaller than $\mu + \mu^2/2$, we can always assume enough initial closeness and that $t_0$ is large enough so that $R_*$ is also large and in particular the last expression is  bounded by $C e^{-(1+\mu') R_*^2/4}$. 

Finally substituting the definition of $R_*$ and using (\ref{eq:phi-tilde-F}), we have
\begin{align*}
|\tilde{F}(T) - F(\Gamma)| \leq C(2(\tilde{F}({T-1}) - \tilde{F}({T+1})) + e^{-T/2})^{\frac{1+\mu'}{2}}+ K_3 e^{-nT/2} + F(\Gamma) ( \exp(K_1e^{-T})-1).
\end{align*}

Since $T\geq t_0$ is large, we may estimate this by 
\begin{align*}
|\tilde{F}(T)-F(\Gamma)| \leq C(\tilde{F}({T-1})- \tilde{F}({T+1}))^{\frac{1+\mu'}{2}} + C' e^{-\frac{1+\mu'}{4}T}.
\end{align*}
\end{proof}

\subsection{Graphical representation}

Again, to apply the {\L}ojasiewicz inequalities we need a good graphical representation of $\Sigma_t$ over the cylinder $\Gamma$. In the non-compact setting, even the initial closeness is nontrivial. As in \cite{colding2015uniqueness}, we use the rigidity of the cylinder to get closeness (to \textit{some} cylinder) at all times:

First, we prove the following analog of \cite[Corollary 0.3]{colding2015rigidity} for MCF with forcing, which establishes uniqueness of tangent type. 
\begin{proposition}\label{CIM03}
If one tangent flow at a singular point of mean curvature flow with forcing is a multiplicity one cylinder, they all are.
\end{proposition}

The proof of this proposition largely follows that of \cite{colding2015rigidity}. The key is to replace their Proposition 2.13.
Instead, we show:

\begin{lemma}\label{CIM Lemma}
Given $n,\lambda_0,\epsilon>0$, there exists $\delta>0$ and $t_0\in(0,\infty)$ so that if $\Sigma_t\subset \R^{n+k}$ is a RMCFf with $\sup_{x\in B_R, r>0} r^{-n} |\Sigma_t\cap B_r(x)| \leq \lambda_0$ for $t\in[T,T+1]$, $T\geq t_0$,  and
\begin{align*}
F(\Sigma_{T})-F(\Sigma_{T+1})\le\delta.
\end{align*}
Then there is an $F$-stationary varifold $\Sigma$ such that $d_V(\Sigma, \Sigma_t)\le \epsilon$ for all $t\in[T,T+1]$ and $F(\Sigma)\le\lambda_0$. 
\end{lemma}

Here we use the metric $d_V$ from \cite[Equation 2.11]{colding2015rigidity}. This metric is defined on finite Radon measures and its induced topology is the weak topology of Radon measures.

\begin{proof}
Suppose the lemma is false. 
Then there exists $T_i\to \infty$ and a sequence $\Sigma^i_t$ of RMCFf's such that $F(\Sigma_{T_i}^i)-F(\Sigma_{T_i+1}^i)\le \frac1i$ and that for every $F$-stationary varifold $\Sigma$, we have
\begin{align}
\label{contradiction}
d_V(\Sigma, \Sigma_{t_i}^i)\ge\epsilon>0
\end{align}
for some $t_i\in[T_i,T_i+1]$.

Let $\tilde{\Sigma}^i_t = \Sigma^i_{t-T_i}$ and $\tilde{M}^i_s=\sqrt{-s} \tilde \Sigma_t^i$, $t=-\ln(-s)$, be the corresponding `unrescaled' flows, which are $e^{-T_i}K$-almost Brakke flows. By the compactness theorem for almost Brakke flows (cf. \cite[Section 11]{white1997stratification}) and using a diagonal sequence, we find that the $\tilde{M}^i_s$ converge to an unforced Brakke flow $\tilde{M}^\infty_s$. The corresponding flow $\tilde{\Sigma}^\infty_t = e^{t/2} \tilde{M}^\infty_s$ is a unforced rescaled MCF, and satisfies $F(\tilde{\Sigma}^\infty_0) - F(\tilde{\Sigma}^\infty_1) =0$.

But then by the monotonicity of $F$ under RMCF, $\tilde{\Sigma}^\infty_t \equiv  \tilde{\Sigma}^\infty$ must be a static RMCF, that is, induced by some $F$-stationary varifold $ \tilde{\Sigma}^\infty$. 
However, this contradicts equation \eqref{contradiction} for $\Sigma= \tilde{\Sigma}^\infty$.
\end{proof}

We now complete the proof of Proposition \ref{CIM03}:

\begin{proof}[Proof of Proposition \ref{CIM03}]

Suppose that one of the tangent flows that $(0,0)$ is cylindrical. By monotonicity of $\tilde{F}$, we have $\tilde{F}(t) \searrow F(\Gamma)$. 
In particular, for any $\delta>0$ there exists $t_0\geq 1$ such that $\tilde{F}(t) -\tilde{F}(t+1) \leq \tilde{F}(t) - F(\Gamma) <\frac{\delta}{3}$ for all $t\geq t_0$. 
Moreover, for $t_0$ sufficiently large, we can ensure that $|\tilde F(t)-F(\Sigma_t)|\le \frac\delta3$ for all $t\ge t_0$.
Then by the triangle inequality, $ F(\Sigma_t)- F(\Sigma_{t+1})\le \delta$ for any $t\ge t_0$.

We may then apply Lemma \ref{CIM Lemma} to deduce the existence of an $F$-stationary varifold $\Sigma$ such that $d_V(\Sigma,\Sigma_t)\le \epsilon$ for all $t\in [T,T+1]$ and $\lambda(\Sigma)\le \lambda_0$. The rigidity of the cylinder \cite[Corollary 2.12]{colding2015rigidity} (see also \cite[Theorem 0.11]{CMcomplexity}), finishes the proof.
\end{proof}

\begin{remark}
The proof above is slightly simpler than the proof of \cite[Theorem 0.2]{colding2015rigidity}.
The simplification arises from the fact that we again focus on the uniqueness of cylindrical tangent flows, corresponding to their Corollary 0.3, rather than their full Theorem 0.2, which may be considered an $\epsilon$-regularity result for tangent flows at nearby points. However, the proof above can also be adapted to also give the $\epsilon$ regularity result for MCF with forcing.
\end{remark}

\subsection{Final uniqueness}

\begin{theorem}
Let $M_\tau^n$ be an embedded MCF with forcing in $\mathcal{U} \subset \mathbb{R}^N$. If one tangent flow at at a singular point is a multiplicity one cylinder, then the tangent flow at that point is unique. That is, any other tangent flow is also a cylinder (with the same axis and multiplicity one). 
\end{theorem}
\begin{proof}
We may assume without loss of generality that the singular point is $(0,0)$, and take $r_0>0$ such that $B_{4r_0} \subset \mathcal{U}$. Then we consider the corresponding RMCFf $\Sigma_t$, to which all the results of this section apply. 

Let $\delta_j = \sqrt{\tilde{F}(j-1) - \tilde{F}(j+2)}$. By equation \eqref{eq:phi-Ftilde} we have that $\left( \int_{j-1}^{j+2} \|\phi\|^2_{L^2(\Sigma_t \cap B_{3e^{t/2}r_0})}dt \right)^\frac{1}{2} < C\delta_j$. 
We proceed as in \cite[Theorem 0.2]{colding2015uniqueness}. By the rigidity of the cylinder, Proposition \ref{CIM03}, any other tangent flow must be induced by a cylinder. 
By White's local regularity, it follows that for any $R$ there exists $t_0$ so that the RMCFf satisfies:
\begin{itemize}
\item[($\dagger$)] For any $T\geq t_0$, there is a cylinder $\Gamma$ so that for all $t\in [T-1,T+1]$, $\Sigma_t\cap B_{R_1}$ is a normal graph over $\Gamma$ with $C^{2,\alpha}$ norm at most $\epsilon_1$. 
\end{itemize}
Since $|\mathbf{G}^\perp|$ is bounded, the $L^1$ distance between time slices of the RMCFf is bounded by $\delta_j + Ke^{-j/2}$. Thus to prove uniqueness, it is enough to show that $\sum_j (\delta_j + e^{-j/2})$ converges. The geometric series $\sum_j e^{-j/2}$ certainly converges, and combining the discrete differential inequality, Theorem \ref{T:discrete}, with Lemma \ref{L:discrete1} and Lemma \ref{L:discrete2} for $f(t) = \tilde{F}(t) - F(\Gamma)$ shows that in fact $\sum_j \delta_j^\beta$ converges for some $\beta<1$. This completes the proof. 
\end{proof}


\appendix

\section{Solving the discrete differential inequality}\label{A: discrete}

\begin{lemma}\label{L:discrete1}
Suppose that $f:[0,\infty)\to [0,\infty)$ is a non-increasing function, and there are constants $\gamma>0$, $K>0$ and $E(t)\geq 0$ so that for $t\geq 1$, we have 
\begin{align*}
  f(t)^{1+\gamma}  \leq K(f(t-1)-f(t+1)) + E(t).
\end{align*}
If $E(t) = o(t^{-\frac{\gamma+1}{\gamma}})$, there exists a constant $C$ so that $f(t)\leq Ct^{-1/\gamma}$ for all $t\ge1$. 
\end{lemma}

\begin{proof}
Following Colding-Minicozzi \cite{colding2015uniqueness}, by scaling and translating $f$, we may assume without loss of generality that $f(0)\in (0,1/2]$ and $K=1$. By the assumption on $E(t)$ there exists $t_1$ so that $t^{\frac{\gamma+1}{\gamma}} E(t-1) \leq \frac{1}{2}f(0)^{1+\gamma}$ for $t\geq t_1$. Now set $t_0 = 2+ \max (t_1, 2^{3+\gamma} f(0)^{-\gamma} \gamma^{-1})$. 

Choose $C$ so that $f(0) =Ct_0^{-1/\gamma}$.
This implies $f(t) \leq C t^{-1/\gamma}$ for all $t\leq t_0$. 
We show by induction on $j$ that this inequality holds for all $t\leq t_0 +2j$. Indeed, suppose this holds for some $j$. 

By the recurrence on $f$ and using $K=1$, we have for $t\geq 2$
\begin{align}\label{eq:a1}
f(t)^{1+\gamma} \leq f(t-1)^{1+\gamma} \leq f(t-2) - f(t) + E(t-1).
\end{align}
 Suppose for the sake of contradiction that $f(t) > Ct^{-1/\gamma}$ for some $t\in (t_0+2j,t_0+2j+2]$. Note that by choice of $t_0$ which implies $t>1$ and $t>t_1$, we have 
 \begin{align}\label{eq:a2}
 C^{-1} t^{1/\gamma} E(t-1) \leq \frac{1}{2} C^{-1} f(0)^{1+\gamma} t^{-1} \leq \frac{1}{2}C^\gamma t^{-1}.
 \end{align}
Then using $(1+h)^{-\gamma} \leq 1-2^{-1-\gamma} \gamma h$ for $h\leq 1$, we have by \eqref{eq:a1}, \eqref{eq:a2} and our choice of $t_0$
\begin{align*}
\begin{split}
f(t-2)^{-\gamma} <& C^{-\gamma}t (1+ C^\gamma t^{-1} -C^{-1} t^{-1/\gamma} E(t-1))^{-\gamma}
\\\leq & C^{-\gamma}t (1+ \frac{1}{2} C^\gamma t^{-1})^{-\gamma} \\\leq & C^{-\gamma}(t- 2^{-2-\gamma}\gamma C^\gamma) \\\leq& C^{-\gamma} (t-2).
\end{split}
\end{align*}
But $f(t-2) \leq C(t-2)^{-1/\gamma}$ by the inductive hypothesis, so this is a contradiction. This completes the induction and the proof of the lemma.
\end{proof}

The following is essentially the content of the proof of Lemma 7.8 in the high codimension paper \cite{colding2019regularity}:

\begin{lemma}\label{L:discrete2}
Suppose that $\sum_{i=j}^\infty \delta_i^2 \leq C j^{-\rho}$ for some $\rho>1$ and some constant $C$. Then there exists $\bar{\alpha}<1$ such that $\sum_{j=1}^\infty \delta_j^{\bar{\alpha}} <\infty$. 
\end{lemma}

Note that if $\gamma<1$, the previous lemma would allow this lemma to be applied with $\delta_j = \sqrt{f(j) - f(j+1)}$ and $\rho = 1/\gamma$.

\section{Evolution of $\phi$}
\label{sec:evol-phi}

\begin{proof}[Proof of Lemma \ref{lem:evol-phi}]
Recall that $\phi = \mathbf{H} + \frac{x^\perp}{2}$ and the RMCFf satisfies $\pr_t x = \phi + e^{-t/2}\mathbf{G}^\perp$ where $\mathbf{G}$ is a fixed ambient vector field.
According to Proposition 1.3 in \cite{colding2019regularity}, for a general submanifold we  have 
\begin{align*}
\begin{split}
LA_{ij} =& A_{ij}+2\langle A_{jl}, A_{ik}\rangle A_{lk} - \langle A_{ml}, A_{il}\rangle A_{jm} \\&- \langle A_{jl}, A_{ml}\rangle A_{im} + \Hess_\phi(e_i,e_j) + \langle \phi, A_{im}\rangle A_{mj},
\end{split}
\end{align*}
where $L=\mathcal L+\frac12+\langle \cdot, A_{kl}\rangle A_{kl}$ with $\mathcal L=\Delta -\frac12\nabla_{x^T}$,
and
\begin{align*}
L\mathbf{H} = \mathbf{H} + \Lap \phi +  \langle \phi, A_{ij}\rangle A_{ij}.
\end{align*}
According to equation 20 in \cite{andrews2010mean}, for a normal flow $\pr_t x = \vec{V}$, the time evolution is
\begin{align*}
\nabla_{\pr_t} A_{ij} = \Hess_{\vec{V}}(e_i,e_j) + A_{ik} \langle \vec{V}, A_{jk}\rangle.
\end{align*}
Taking the trace gives $\nabla_{\pr_t} \mathbf{H} = \Lap (\vec{V}) + A_{ij}\langle A_{ij}, \vec{V}\rangle.$ In our setting this becomes 
\begin{align*}
 \nabla_{\pr_t} \mathbf{H} = \Lap \phi + \langle A_{ij},\phi\rangle A_{ij} + e^{-t/2}(\Lap {\mathbf G}^\perp + \langle A_{ij},{\mathbf G}^\perp\rangle A_{ij}).
\end{align*}

Consider the normal projection $\Pi$. We use $\Pi'$ to denote its derivative in either a spatial or time direction. Differentiating $\Pi^2=\Pi$ implies that $\Pi' \circ \Pi + \Pi\circ \Pi'=\Pi'$. Composing with $\Pi$ on both sides then gives $\Pi \circ \Pi' \circ \Pi=0$. 
Also by symmetry we have $\langle \Pi' \circ \Pi (V) ,e_j \rangle = \langle \Pi(V), \Pi'(e_j)\rangle$. 
For the spatial derivatives, differentiating $\Pi(e_i)=0$ gives 
\begin{align*}
(\nabla_j \Pi)(e_i) = -\Pi(\nabla_j e_i)  = - A_{ij}.
\end{align*}
 It follows that $\nabla_i \Pi(x) = - A(e_i, x^T) - e_j \langle A_{ij}, x^\perp\rangle$. One can follow \cite[Lemma 2.7]{colding2019regularity} in normal coordinates to find that $\nabla^\perp_k \nabla^\perp_i \Pi(x) = -(\nabla A)(e_i,e_k,x^T) - A(e_i, \nabla_j^T x^T)$. Note that $\langle e_k, \nabla_j^T x^T\rangle = g_{jk} - \langle e_k, \nabla_j x^\perp\rangle = g_{jk} + \langle A_{jk},x^\perp\rangle$. Therefore,
\begin{align*}
\nabla^\perp_k \nabla^\perp_i \Pi(x) = -(\nabla A)(e_i,e_k,x^T) - A_{ij} - A_{ik} \langle A_{jk}, x^\perp\rangle.
\end{align*}
 Taking the trace gives $ \Lap \Pi(x) = -\nabla^\perp_{x^T} \mathbf{H} - \mathbf{H} - \langle x^\perp, A_{ij}\rangle A_{ij}$ and adding the lower order terms gives
\begin{align*} 
Lx^\perp = -\nabla^\perp_{x^T}\phi -\mathbf{H} + \frac{1}{2}x^\perp.
\end{align*}
For the time derivative, differentiating $\Pi(e_i)=0$ and commuting the time derivative gives $(\pr_t \Pi) (e_i) = -\Pi(\pr_t e_i) = -\nabla_i^\perp (\phi + e^{-t/2} \mathbf{G}^\perp)$. Using the spatial derivative, this becomes 
\begin{align*}
(\pr_t \Pi)(e_i) = -\nabla_i^\perp \phi - e^{-t/2}(-A(e_i,\mathbf{G}^T) + \nabla_i^\perp \mathbf{G}).
\end{align*}
 In particular,
\begin{align*}
\begin{split}
\pr_t (\Pi(x)) =& (\pr_t \Pi)(x) + \Pi(\pr_t x) =\phi + e^{-t/2} \mathbf{G}^\perp  - \nabla^\perp_{x^T} \phi\\& + e^{-t/2} A(x^T, \mathbf{G}^T) - e_j \langle x^\perp, \nabla_j^\perp \phi - e^{-t/2}(A(e_j, \mathbf{G}^T) + \nabla_j^\perp \mathbf{G})\rangle.
\end{split}
\end{align*}
Combining this we obtain
\begin{align*}
(\nabla_{\pr_t} - L) \phi = e^{-t/2} (\Lap \mathbf{G}^\perp + \langle \mathbf{G}^\perp, A_{ij}\rangle A_{ij} + \frac{1}{2} A(x^T, \mathbf{G}^T) - \frac{1}{2} \nabla^\perp_{x^T} \mathbf{G} + \frac{1}{2}\mathbf{G})
\end{align*}
which finishes the proof.
\end{proof}

\section{Interior estimates for MCF with forcing term}

\label{sec:interior-estimates}

Let $\mathbf{F}$ be a smooth vector field on $\R^N$ and $x:I\times M^n\to\R^N$ a smooth family of embeddings which satisfy
\begin{align*}
\partial_t x=\mathbf{H}+{\mathbf{F}}^\perp.
\end{align*}
Our goal of this section is to prove interior estimates for this flow. 
For this purpose we begin with computing the evolution equations.

\begin{proposition}
We have
\begin{equation*}
\begin{split}
&(\partial_t-\Delta)|\nabla^k A|^2\\
 =&-2|\nabla^{k+1}A|^2+\nabla^kA \ast\nabla^{k+2}{\mathbf{F}}^\perp
\\
&+\nabla^kA\ast\sum_{\substack{i_1+i_2+i_3=k}} \nabla^{i_1}A\ast\nabla^{i_2}A\ast\nabla^{i_3}A  +\nabla^kA\ast\sum_{\substack{i_1+i_2+i_3=k}} \nabla^{i_1}A\ast\nabla^{i_2}A\ast\nabla^{i_3}{\mathbf{F}}^\perp.
\end{split}\end{equation*}
\end{proposition}

Note that the $\nabla\mathbf{F}^\perp$ terms may be related to Euclidean derivatives of $\mathbf{F}$ by
\begin{align*}
\nabla^k\mathbf F^\perp=\sum_{j_1+\dots+j_a+a+b=k}\nabla^{j_1}A\ast\dots\ast \nabla^{j_a}A\ast D^b\mathbf F.
\end{align*}

\begin{proof}
From the timelike Codazzi equations, cf. \cite[Equation (18)]{andrews2010mean}, we obtain
\begin{align*}
\partial_t A_{ij}=&  \nabla^2_{ij} ( \mathbf{H} + {\mathbf{F}}^\perp) + \langle \mathbf{H}+{\mathbf{F}}^\perp,  A_{ik} A_{jk}\rangle 
\end{align*} 
where $\{e_i\}_{i=1,\dots, n}$ are an orthonormal frame of $TM$. By Simons' identity, cf. \cite[Equation (23)]{andrews2010mean}, we have
\begin{align*}
(\partial_t-\Delta)A=\nabla^2{\mathbf{F}}^\perp+A\ast A\ast A+{\mathbf{F}}^\perp\ast A\ast A.
\end{align*}
Applying the lemma below, we inductively obtain
\begin{equation*}\begin{split}
\partial_t\nabla^kA=&\Delta \nabla^kA+\nabla^{k+2}{\mathbf{F}}^\perp \\
&+\sum_{\substack{i_1+i_2+i_3=k}} \nabla^{i_1}A\ast\nabla^{i_2}A\ast\nabla^{i_3}A  +\sum_{\substack{i_1+i_2+i_3=k}} \nabla^{i_1}A\ast\nabla^{i_2}A\ast\nabla^{i_3}{\mathbf{F}}^\perp.
\end{split}\end{equation*}
Next, we note that
\begin{align*}
\Delta|\nabla^{k}A|^2=2\langle\Delta\nabla^kA,\nabla^kA\rangle+2|\nabla^{k+1}A|^2
\end{align*}
and the inverse metric evolves as 
\begin{align*}
\partial_t g^{ij}=2\langle \mathbf{H}+{\mathbf{F}}^\perp, A_{ij} \rangle.
\end{align*}
Combining all the above identities yields the proposition. 
\end{proof}

\begin{lemma}
Let $S$ and $T$ be tensors satisfying the evolution equation
\begin{align*}
\partial_t S=\Delta S+T,
\end{align*}
then the covariant derivative $\nabla S$ satisfies an equation of the form 
\begin{align*}
\partial_t \nabla S=\Delta \nabla S+A\ast A\ast\nabla S+A\ast \nabla A\ast S+\nabla T.
\end{align*}
\end{lemma}

\begin{proof}
Lemma 13.1 in \cite{hamilton82three} states
\begin{align*}
\partial_t\nabla S=\Delta\nabla S+\Rm\ast \nabla S+S\ast\nabla \Rm+\nabla T.
\end{align*}
Hence, the result follows from the Gauss and Codazzi equations.
\end{proof}

To state the interior estimates it will be convenient to define $r(\textbf x,t)=|\textbf x|^2+2nt$. As in \cite[Theorem 3.7]{ecker1991interior}, we obtain:

\begin{theorem}
Let $R>0$ be such that $\{x\in M_t:r(\textbf x,t)\le R^2\}$ is compact for $t\in[0,T]$. Then for $0\le \theta<1$, $t\in[0,T]$ and any integers $l, m\ge0$, we have
\begin{align*}
\sup_{\textbf x\in M_t:r(\textbf x,t)\le \theta R^2}|\nabla^{m+l}A|^2\le C_l t^{-l}
\end{align*}
where
\begin{align*}
C_l=C_l\left(K_l,m,n,N,\theta, \sup_{\textbf x\in M_s:r(\textbf x, s)\le R^2,\,s\in[0,t]}\sum_{i=0}^m|\nabla^iA|^2\right)
\end{align*}
where 
\begin{align*}
K_l:=\sum_{k=0}^{l+m+2}\|D^k{\mathbf{F}}\|_{C^0(\R^N)}.
\end{align*}
\end{theorem}

\begin{proof}
The proof follows essentially as in \cite{ecker1991interior}, with some modifications to handle the forcing term. One proceeds by induction on $l$: Assume that for all $k\le l$ we have
\begin{align*}
\sup_{\textbf x\in M_t:r(\textbf x,t)\le \theta R^2}|\nabla^{m+k}A|^2\le  C_k\psi^{-k}
\end{align*}
where $C_k$ is defined as above and $\psi(t)=\frac{R^2t}{R^2+t}$.

It follows immediately from the evolution equations for $A$ and its derivatives that 

\begin{equation*}
\begin{split}
&(\partial_t-\Delta) |\nabla^{m+l+1} A|^2
\\\le&-2|\nabla^{m+l+2}A|^2+C|\nabla^{m+l+1}A|\left(\sum_{j_1+\dots+j_a+a+b=m+l+3}|\nabla^{j_1}A|\dots|\nabla^{j_a}A|| D^b\mathbf F|\right)\\
&+C|\nabla^{m+l+1}A|\sum_{\substack{i_1+i_2+i_3=m+l+1}} |\nabla^{i_1}A||\nabla^{i_2}A||\nabla^{i_3}A | \\&+C|\nabla^{m+l+1}A|\sum_{\substack{i_1+i_2+i_3=m+l+1}}| \nabla^{i_1}A||\nabla^{i_2}A|\left(\sum_{j_1+\dots+j_a+a+b=i_3}|\nabla^{j_1}A|\dots|\nabla^{j_a}A|| D^b\mathbf F|\right).
\end{split}
\end{equation*}

Using Young's inequality and then the induction hypotheses on the last two terms, one can estimate all derivatives of $A$ up to order $m$ with powers of $\psi^{-1}$
; the highest degree that appears is $m+l+1$. One proceeds similarly for the second term, using the first term to absorb the highest order derivative $|\nabla^{m+l+2}A|$. (Note that the second and fourth terms arise from the forcing term for the flow.) Ultimately, this yields
\begin{align*}
(\partial_t-\Delta)|\nabla^{m+l+1} A|^2
\le&-\frac32|\nabla^{m+l+2}A|^2+\tilde C|\nabla^{m+l+1}A|^2+\tilde C\psi^{-m-1}
\end{align*}
for some constant $\tilde C$ depending on $C_k$, $k\le l$, and $K_{l+1}$.
Similarly, we obtain
\begin{align*}
(\partial_t-\Delta)|\nabla^{m+l}A|^2\le&-\frac32|\nabla^{m+l+1}A|^2+\tilde C\psi^{-m}.
\end{align*}
The remainder of the proof follows exactly as in \cite{ecker1991interior}, by applying the maximum principle to the same test function
\begin{align*}
f:=\psi^{m+1}|\nabla^{m+l+1}A|^2(\Lambda+\psi^{m}|\nabla^{m+l}A|^2),
\end{align*}
for some large constant $\Lambda$ depending on $\tilde C$ and $K_{l+1}$.
\end{proof}

\bibliographystyle{alpha}
\bibliography{HirschZhu.bib}

\begin{thebibliography}{CM19b}

\bibitem[AB10]{andrews2010mean}
Ben Andrews and Charles Baker.
\newblock Mean curvature flow of pinched submanifolds to spheres.
\newblock {\em J. Differential Geom.}, 85(3):357--395, 2010.

\bibitem[BS09]{brendle2009manifolds}
Simon Brendle and Richard Schoen.
\newblock Manifolds with {$1/4$}-pinched curvature are space forms.
\newblock {\em J. Amer. Math. Soc.}, 22(1):287--307, 2009.

\bibitem[CIM15]{colding2015rigidity}
Tobias~Holck Colding, Tom Ilmanen, and William~P. Minicozzi, II.
\newblock Rigidity of generic singularities of mean curvature flow.
\newblock {\em Publ. Math. Inst. Hautes \'{E}tudes Sci.}, 121:363--382, 2015.

\bibitem[CL20]{chodosh2020generalized}
Otis Chodosh and Chao Li.
\newblock Generalized soap bubbles and the topology of manifolds with positive
  scalar curvature.
\newblock {\em arXiv preprint arXiv:2008.11888}, 2020.

\bibitem[CM14]{CMein}
Tobias~Holck Colding and William~P. Minicozzi, II.
\newblock On uniqueness of tangent cones for {Einstein} manifolds.
\newblock {\em Invent. Math.}, 196(3):515--588, 2014.

\bibitem[CM15a]{CMsurvey}
Tobias~Holck Colding and William~P. Minicozzi, II.
\newblock {{\L ojasiewicz}} inequalities and applications.
\newblock In {\em Surveys in differential geometry 2014. {R}egularity and
  evolution of nonlinear equations}, volume~19 of {\em Surv. Differ. Geom.},
  pages 63--82. Int. Press, Somerville, MA, 2015.

\bibitem[CM15b]{colding2015uniqueness}
Tobias~Holck Colding and William~P. Minicozzi, II.
\newblock Uniqueness of blowups and \l ojasiewicz inequalities.
\newblock {\em Ann. of Math. (2)}, 182(1):221--285, 2015.

\bibitem[CM16a]{CMdiff}
Tobias~Holck Colding and William~P. Minicozzi, II.
\newblock Differentiability of the arrival time.
\newblock {\em Comm. Pure Appl. Math.}, 69(12):2349--2363, 2016.

\bibitem[CM16b]{CMsing}
Tobias~Holck Colding and William~P. Minicozzi, II.
\newblock The singular set of mean curvature flow with generic singularities.
\newblock {\em Invent. Math.}, 204(2):443--471, 2016.

\bibitem[CM18a]{CMreg}
Tobias~Holck Colding and William~P. Minicozzi, II.
\newblock Regularity of the level set flow.
\newblock {\em Comm. Pure Appl. Math.}, 71(4):814--824, 2018.

\bibitem[CM18b]{CMwand}
Tobias~Holck Colding and William~P Minicozzi, II.
\newblock Wandering singularities.
\newblock {\em arXiv preprint arXiv:1809.03585}, 2018.

\bibitem[CM19a]{CMarnold}
Tobias~Holck Colding and William~P. Minicozzi, II.
\newblock Arnold-{Thom} gradient conjecture for the arrival time.
\newblock {\em Comm. Pure Appl. Math.}, 72(7):1548--1577, 2019.

\bibitem[CM19b]{CMcomplexity}
Tobias~Holck Colding and William~P Minicozzi, II.
\newblock Complexity of parabolic systems.
\newblock {\em arXiv preprint arXiv:1903.03499}, 2019.

\bibitem[CM19c]{colding2019regularity}
Tobias~Holck Colding and William~P Minicozzi, II.
\newblock Regularity of elliptic and parabolic systems.
\newblock {\em arXiv preprint arXiv:1905.00085}, 2019.

\bibitem[DO20]{deruelle2020ojasiewicz}
Alix Deruelle and Tristan Ozuch.
\newblock A {\l}ojasiewicz inequality for ale metrics.
\newblock {\em arXiv preprint arXiv:2007.09937}, 2020.

\bibitem[Ede20]{edelen}
Nick Edelen.
\newblock The free-boundary {B}rakke flow.
\newblock {\em J. Reine Angew. Math.}, 758:95--137, 2020.

\bibitem[EH91]{ecker1991interior}
Klaus Ecker and Gerhard Huisken.
\newblock Interior estimates for hypersurfaces moving by mean curvature.
\newblock {\em Invent. Math.}, 105(3):547--569, 1991.

\bibitem[Fee19]{F19}
Paul M.~N. Feehan.
\newblock Resolution of singularities and geometric proofs of the {{\L
  ojasiewicz}} inequalities.
\newblock {\em Geom. Topol.}, 23(7):3273--3313, 2019.

\bibitem[Fee20]{F20}
Paul M.~N. Feehan.
\newblock On the {Morse-Bott} property of analytic functions on {Banach} spaces
  with {{\L ojasiewicz}} exponent one half.
\newblock {\em Calc. Var. Partial Differential Equations}, 59(2):Paper No. 87,
  50, 2020.

\bibitem[Gro20]{gromov2020no}
Misha Gromov.
\newblock No metrics with positive scalar curvatures on aspherical 5-manifolds.
\newblock {\em arXiv preprint arXiv:2009.05332}, 2020.

\bibitem[Ham82]{hamilton82three}
Richard~S. Hamilton.
\newblock Three-manifolds with positive {R}icci curvature.
\newblock {\em J. Differential Geometry}, 17(2):255--306, 1982.

\bibitem[HI01]{huisken2001inverse}
Gerhard Huisken and Tom Ilmanen.
\newblock The inverse mean curvature flow and the {R}iemannian {P}enrose
  inequality.
\newblock {\em J. Differential Geom.}, 59(3):353--437, 2001.

\bibitem[Ilm95]{ilmanen1995singularities}
Tom Ilmanen.
\newblock Singularities of mean curvature flow of surfaces.
\newblock {\em preprint}, 1995.

\bibitem[LM20]{liokumovich2020waist}
Yevgeny Liokumovich and Davi Maximo.
\newblock Waist inequality for 3-manifolds with positive scalar curvature.
\newblock {\em arXiv preprint arXiv:2012.12478}, 2020.

\bibitem[Man14]{Mantoulidis}
Christos Mantoulidis.
\newblock T. {H.} {Colding} and {W.} {P.} {Minicozzi}'s ”{Uniqueness} of
  blowups and {Lojasiewicz} inequalities”.
\newblock {\em unpublished lecture notes}, 2014.

\bibitem[Per02]{perelman2002entropy}
Grisha Perelman.
\newblock The entropy formula for the ricci flow and its geometric
  applications.
\newblock {\em arXiv preprint math/0211159}, 2002.

\bibitem[Sch14]{schulze2014uniqueness}
Felix Schulze.
\newblock Uniqueness of compact tangent flows in mean curvature flow.
\newblock {\em J. Reine Angew. Math.}, 690:163--172, 2014.

\bibitem[Sim83]{simon1983asymptotics}
Leon Simon.
\newblock Asymptotics for a class of non-linear evolution equations, with
  applications to geometric problems.
\newblock {\em Annals of Mathematics}, pages 525--571, 1983.

\bibitem[SZ20]{sun2020rigidity}
Ao~Sun and Jonathan~J. Zhu.
\newblock Rigidity and {\l}ojasiewicz inequalities for clifford self-shrinkers.
\newblock {\em arXiv preprint arXiv:2011.01636}, 2020.

\bibitem[Whi97]{white1997stratification}
Brian White.
\newblock Stratification of minimal surfaces, mean curvature flows, and
  harmonic maps.
\newblock {\em J. Reine Angew. Math.}, 488:1--35, 1997.

\bibitem[Whi05]{white2005local}
Brian White.
\newblock A local regularity theorem for mean curvature flow.
\newblock {\em Ann. of Math. (2)}, 161(3):1487--1519, 2005.

\bibitem[Zhu20]{zhu2020ojasiewicz}
Jonathan~J. Zhu.
\newblock {\L}ojasiewicz inequalities, uniqueness and rigidity for cylindrical
  self-shrinkers.
\newblock {\em arXiv preprint arXiv:2011.01633}, 2020.

\end{thebibliography}

\end{document}